\documentclass[usletter]{amsart} 

\usepackage{amsmath,amsthm,amssymb,amsfonts,mathrsfs,color,hyperref, mathtools,crop, graphicx, enumitem, todonotes}
\usepackage[left=3cm, right=3cm]{geometry}

\theoremstyle{plain}
\begingroup
\newtheorem{theorem}{Theorem}[section]
\newtheorem*{theorem*}{Theorem}
\newtheorem*{"theorem"}{``Theorem''}
\newtheorem{corollary}[theorem]{Corollary}

\newtheorem{lemma}[theorem]{Lemma}
\endgroup

\theoremstyle{definition}
\begingroup
\newtheorem{definition}[theorem]{Definition}
\endgroup

\theoremstyle{remark}
\begingroup
\newtheorem{remark}[theorem]{Remark}
\newtheorem{example}[theorem]{Example}
\endgroup 

\numberwithin{equation}{section}
\newenvironment{pde}{\left\{\begin{array}{rll} } {\end{array}\right.}

\newcommand{\N}{\mathbb N} 
 
\newcommand{\Z}{\mathbb Z} 
\newcommand{\R}{\mathbb R} 
\newcommand{\E}{{\mathbb E}}

\renewcommand{\L}{{\mathcal L}}

\newcommand{\C}{{\mathcal C}}

\newcommand{\LRa} {\Leftrightarrow}
\newcommand{\Ra} {\Rightarrow}

\renewcommand{\d}{\,\mathrm{d}}

\newcommand{\dx}{\,\mathrm{d}x}
\newcommand{\dy}{\,\mathrm{d}y}
\newcommand{\dz}{\,\mathrm{d}z}
\newcommand{\ds}{\,\mathrm{d}s}
\newcommand{\dt}{\,\mathrm{d}t}
\newcommand{\dr}{\,\mathrm{d}r}

\newcommand{\ReLU}{{\mathrm{ReLU}}}

\newcommand{\eps}{\varepsilon}
\newcommand{\average}{{\mathchoice {\kern1ex\vcenter{\hrule height.4pt
width 6pt depth0pt} \kern-9.7pt} {\kern1ex\vcenter{\hrule
height.4pt width 4.3pt depth0pt} \kern-7pt} {} {} }}

\allowdisplaybreaks

 \makeatletter
\@namedef{subjclassname@2020}{2020 Mathematics Subject Classification}
\makeatother

\newcommand{\B}{\mathcal{B}}
\renewcommand{\Re}{{\mathrm{Re}}}
\renewcommand{\Im}{{\mathrm{Im}}}

\begin{document}

\title[Shallow ReLU$^\alpha$-networks and the Dirichlet Laplacian]{Solving the Poisson Equation with Dirichlet data by shallow ReLU$^\alpha$-networks: A regularity and approximation perspective}
%{Representing solutions to elliptic PDEs by neural networks:\\{\small Shallow networks, homogeneous activation functions, \\ and the Dirichlet Laplacian}}

\author{Malhar Vaishampayan}
\address{Malhar Vaishampayan\\
University of Pittsburgh,
Department of Mathematics,
Thackeray Hall,
Pittsburgh, PA 15213
}
\email{msv17@pitt.eduu}

\author{Stephan Wojtowytsch}
\address{Stephan Wojtowytsch\\
University of Pittsburgh,
Department of Mathematics,
Thackeray Hall,
Pittsburgh, PA 15213
}
\email{s.woj@pitt.edu}

\date{\today}

\subjclass[2020]{}

\begin{abstract}
For several classes of neural PDE solvers (Deep Ritz, PINNs, DeepONets), the ability to approximate the solution or solution operator to a partial differential equation (PDE) hinges on the abilitiy of a neural network to approximate the solution in the spatial variables. We analyze the capacity of neural networks to approximate solutions to an elliptic PDE assuming that the boundary condition can be approximated efficiently. Our focus is on the Laplace operator with Dirichlet boundary condition on a half space and on neural networks with a single hidden layer and an activation function that is a power of the popular ReLU activation function.
\end{abstract}

\maketitle

%\setcounter{tocdepth}{1}
% \tableofcontents

\section{Introduction}

Following their success in various fields of data science and machine learning, neural networks have more recently enjoyed success in areas of scientific computing. Various model approaches have been introduced, among them models which solve a single partial differential equation (PDE), such as the Deep Ritz method \cite{yu2018deep, muller2019deep, duan2021convergence, jiao2021error, lu2021priori, dondl2022uniform} and physics-informed neural networks (PINNs, \cite{raissi2019physics, cai2021physics, cuomo2022scientific}), and models which learn the solution operator of a PDE, such as Deep Operator Networks (DeepONets, \cite{lu2019deeponet}) and Neural Operators (e.g.\ Fourier Neural Operators, Low Rank Neural Operators, Graph Neural Operators, \cite{li2020fourier, kovachki2023neural, azizzadenesheli2024neural}). For all of these models, the literature goes far beyond the listed sources.

The Deep Ritz method, PINNs and DeepONets represent the solution $u$ of the PDE by a neural network in the spatial variables $x$ or in space-time variables $(t,x)$. It is therefore important to study the question: In a given classes of partial differential equations, can the solution to a PDE in the class be represented (or approximated efficiently) by a neural networks? 

Classical approaches to solving elliptic second order PDEs split this question into three components:
\begin{enumerate}
    \item Given a model PDE, what is the regularity of solutions? (regularity theory)
    
    \item Given a function in a certain regularity class, how well can its elements be approximated using a model class of parametrized functions? (approximation theory)

    \item Given the solution of a surrogate problem in a model class, how does it compare to the best approximation in the model class to the solution to the original problem? (numerical analysis)
\end{enumerate}

In this work, we approach the first two of these questions for the model problem of the Dirichlet Laplacian:
\begin{equation}\label{eq intro pde}
\begin{pde}
    -\Delta u &= 0 &\text{in }\Omega\\
    u &= g &\text{in }\partial\Omega
\end{pde}
\end{equation}
on the half-plane $\Omega = \R^2_+ = \{(x,y)\in\R^2 : y>0\}$. The function classes which we consider are `Barron spaces': Function spaces specifically designed in such a way that functions in them can be approximated efficiently by shallow neural networks.

A feed-forward neural network (or `multi-layer perceptron') is a function with the architecture
\[
f(x) = T^L \circ \sigma \circ T^{L-1}\circ \sigma \circ\dots \circ \sigma \circ T^1 (x)
\]
where $T^\ell :\R^{d_{\ell-1}}\to \R^{d_\ell}$ is an affine linear map between Euclidean maps and by abuse of notation we consider $\sigma :\R^d\to\R^d$ given by $\sigma(z) = (\sigma(z_1), \dots, \sigma(z_d))$ for any $\sigma:\R\to\R$ and any $d\in\mathbb N$. The parameters of the neural network are the matrix elements of the linear part of the affine maps (`weights') and the vector entries of the affine part (`biases'). The function $\sigma :\R\to\R$ is called the `activation function' of the neural network and generally assumed to be non-polynomial.

A shallow neural (or `perceptron') is a neural network of depth $L= 2$, i.e.\ with two `layers' $T_1, T_2$. We often write it as
\[
f(x) = \sum_{i=1}^n a_i\,\sigma(w_i^Tx+b_i) + B, \qquad w_i \in\R^d, \:\:a_i, b_i, B\in \R
\]
and commonly assume that $B=0$. The focus of this article is on the shallow case.

One of the most popular activation functions in modern deep learning is the rectified linear unit (ReLU) function $\sigma(z) = \max\{z,0\}$. It enjoys several major benefits: From a practical perspective, it can be evaluated significantly faster than more involved activation functions like $\tanh$, and it mitigates the `vanishing gradients phenomenon' since its derivative does not approach zero at infinity. Additionally, the positive homogeneity of $\sigma$ facilitates a theoretical analysis in several ways.

For neural PDE solvers, however, $\sigma$ may not be the function of choice. The strong form of a second order PDE for instance cannot be evaluated on the piecewise linear functions represented by a ReLU-network, for which the second order derivatives are measures localized on the cell edges where the network is non-smooth. We therefore consider more generally powers $\sigma_\alpha(z) = \sigma(z)^\alpha$ of the ReLU activation (also sometimes referred to as RePU or rectified power unit activations).

The function class $\Sigma_m$ of shallow neural networks of `width' $m$ is a non-linear subset of various function spaces (depending on the activation function $\sigma$). It is well known that the closure of $\bigcup_{m=0}^\infty \Sigma_m$ in e.g.\ $C^0(K)$ coincides with the whole space $C^0(K)$ \cite{cybenko1989approximation, hornik1991approximation, hornik1993some, park2024qualitative}. However, approximating a general continuous function requires very large `weights' $(a_i, w_i, b_i)$, which complicates both the training of networks and finding statistical generalization bounds \cite{wojtowytsch2020can}. Barron function classes correspond to finding a closure while keeping the coefficients bounded in a suitable sense. The natural bound depends on the choice of $\sigma$. For $\sigma_\alpha$, the potentially most natural controlled quantity is $\sum_{i=1}^m |a_i| \big(|w_i|+|b_i|\big)^\alpha$.

Solving elliptic PDEs in Barron spaces has been considered previously for instance in \cite{chen2021representation}. Schr\"odinger eigenvalue problems with periodic boundary conditions have been studied in the related `spectral Barron class' \cite{lu2022priori}.

In \cite{wojtowytsch2020some}, the authors study the Laplace equation both on the whole space (with non-zero right hand side) and in a finite domain (with right hand side zero). On the whole space, it is easy to see that a solution $u(x) = \frac{a}{(\alpha+2)(\alpha+1)\,\|w\|^2}\,\sigma_{\alpha+2}(w^Tx+b)$ to the equation $\Delta u = a\,\sigma_\alpha(w^Tx+b)$ can be represented by a single neuron activation with a higher power $\alpha+2$ of ReLU, but that it cannot be approximated by any finite superposition
\[
u_\alpha(x)= \sum_{i=1}^n a_i\,\sigma_\alpha(w_i^Tx+b_i)
\]
globally since the right hand side grows at most as $|x|^\alpha$ at infinity while any solution to the PDE has to grow at least as $|x|^{\alpha+2}$. When changing the activation function from $\sigma_\alpha$ (representing the right hand side $f$) to $\sigma_{\alpha+2}$ (representing the solution $u$), complications remain due to the fact that the natural `norm' or coefficient bound for the solution network scales as
\[
\frac{|a|}{|w|^2}\big(|w|+|b|\big)^{\alpha+2} \sim |a| \,|w|^\alpha + |a| \,\frac{|b|^{2+\alpha}}{|w|^2}\qquad\text{rather than}\quad |a|\big(|w|+|b|\big)^\alpha,
\]
i.e.\ the control of the magnitude of the bias variable is mismatched.
For ReLU activation, a version of Barron spaces in which the magnitude of the bias is unconstrained is available with subtle differences between the two function classes \cite{wojtowytsch2022optimal, boursier2023penalising}. To the best of our knowledge, this has not been extended to $\ReLU^\alpha$-activation.

There is no difference between the two function classes when we only consider the functions on {\em bounded} sets, but other difficulties may arise. In \cite{wojtowytsch2020some}, the authors demonstrate that harmonic functions with boundary values in Barron space are generally {\em not} Barron functions up to the boundary. The counterexample only applies in dimension $d\geq 3$, and a much more detailed analysis can be found in a forth-coming article \cite{pde_relu}. In particular, we can see that, heuristically speaking, a harmonic function with boundary values in Barron space only fails to be Barron-regular `by a logarithm'. This resembles the logarithmic failure of regularity for instance of harmonic functions with Lipschitz-continuous boundary values \cite{hile1999gradient}.

It is well-known that elliptic regularity theory fails in certain function classes of integer order, such as spaces of continuously differentiable functions $C^k = C^{k,0}$ and spaces of functions with Lipschitz-continuous derivatives $C^{k,1}$. As a remedy, spaces of H\"older-continuous functions $C^{k,\alpha}$ are considered with a fractional smoothness parameter $\alpha\in(0,1)$. Analogously, we can ask:

\begin{enumerate}
    \item Is there a qualitative difference between the regularity theory for \eqref{eq intro pde} in Barron spaces associated to $\ReLU^\alpha$-activation for integer and non-integer powers $\alpha$?

    \item For integer powers $\alpha = k$, is the failure of regularity merely logarithmic also for $k\geq 2$?
\end{enumerate}

We answer the first question in the affirmative when contrasting $\alpha\in(0,1)$ with $\alpha \in\mathbb N$, but it remains open for $\alpha\in(1,\infty)\setminus \mathbb N$.

The second question is possibly more important from the standpoint of applications since $\sigma_\alpha$ loses one of its main advantages when $\alpha\notin\N$: The speed with which $\sigma_\alpha$ and $\sigma_\alpha'$ can be evaluated is much higher than that for other activation functions only if $\alpha$ is a (small) integer. We give a positive answer also here (in a sense made precise below), for any $\alpha = k\in\mathbb N$.

While we present our analysis in the half-plane $\R^2_+$, the results remain valid more generally in half-spaces $\R^d_+ =\{ (\hat x, x_d) \in\R^d : x_d>0\}$ since the boundary condition $\sigma_\alpha(w^Tx+b)$ is {\em one-dimensional}. Since the Laplacian is rotationally invariant, we may assume without loss of generality that $w = (w_1, 0, \dots, 0)$. We may then use $u_d (x) = u_2(w_1x_1+b, w_1x_d)$ where
\[
\begin{pde}
-\Delta u_2 &= 0 &\text{in }\R^2_+\\
u &= \sigma_k(x_1) &\text{on }\partial\R^2_+
\end{pde}
\qquad\Ra\quad 
\begin{pde}
-\Delta u_d &= 0 &\text{in }\R^d_+\\
u &= \sigma_k(w^Tx+b) &\text{on }\partial\R^d_+.
\end{pde}
\]
General Barron functions are (continuous) superpositions of functions of this type.

The article is organized as follows. In Section \ref{section barron background}, we review the theory of Barron spaces and their relationships with classical function spaces. In Section \ref{section pde}, we analyze \eqref{eq intro pde} in Barron spaces for small non-integer powers of ReLU activation (Section \ref{section fractional alpha}) and for integer powers of ReLU (Section \ref{section integer alpha}). Open problems are collected in Section \ref{section conclusion}. Some technical details are postponed until the appendix.

\subsection{Concepts and Notation}
We use spaces of Sobolev functions $W^{k,p}$ and H\"older continuous functions $C^{k,\alpha}$ as introduced e.g.\ in \cite{dobrowolski2010angewandte} or other textbooks without further comment.

All measures on Euclidean spaces are measures on the Borel $\sigma$-algebra. If $\mu$ is a measure on $\mathcal X$ and $f:\mathcal X\to \mathcal Y$ is a measurable map, we define the push-forward measure $f_\sharp\mu$ on $\mathcal Y$ by $f_\sharp\mu(B) = \mu(f^{-1}(B))$. 

Lebesgue measure on $\R^d$ is denoted by $\L^d$. If a measure $\mu_1$ has a density $\rho$ with respect to a second measure $\mu_2$, we also write $\mu_1 = \rho\cdot\mu_2$ as short-hand.

For future use, we recall {\em Liouville's Theorem} in a slightly stronger version than commonly presented in courses on partial differential equations. While it is well-known to experts in the field, we sketch the proof for the readers convenience.

Denote $\R^n_+ = \{x\in\R^n : x_n>0\}$ and $x = (\hat x, x_n)$. We write $B_r(x)$ for the open Euclidean ball of radius $r$ centered at $x$ and $B_r:= B_r(0)$.

\begin{theorem}[Liouville] \label{theorem liouville}
Let $u:\R^n \to \R$ be a harmonic function or $u\in C^2(\R^n_+) \cap C^0(\overline{\R^n_+})$ harmonic in $\R^n_+$ a function with boundary values $u(\hat x, 0) = 0$ on $\partial\R^n_+$. 

If there exists $k\in \N$ such that
\[
\lim_{r\to \infty}\frac{\sup_{x\in B_r(0)} |u(x)|}{r^{k+1}} = 0,
\]
then $u$ is a harmonic polynomial of degree at most $k$.
\end{theorem}

\begin{proof}
    The proof combines the ideas of \cite[Theorem 2.2.8]{evans} and  \cite[Exercise 2.5.9]{evans}. Namely, if $u$ is defined on the entire space and $\alpha$ is a multi-index of order $k+1$, then using \cite[Theorem 2.2.7]{evans}, we find that for all $r>0$:
    \[
    |D^{\alpha}\,u(x)| \leq \frac{C_k}{r^{n+k+1}} \int_{B_r(x)} |u(y)|\dy \leq \frac{C_k}{r^{n+k+1}} \int_{B_r(x)} \sup_{z\in B_{r + \|x\|}} |u(z)| \dy = C_k\,\frac{\omega_nr^n}{r^n} \,\frac{\sup_{z\in B_{r + \|x\|}} |u(z)|}{r^{k+1}}
    \]
    where $\omega_n$ denotes the volume of the $n$-dimensional unit ball. For $r\geq \|x\|$, we can replace
    \[
    \frac{\sup_{z\in B_{r + \|x\|}} |u(z)|}{r^{k+1}} \leq \frac{\sup_{z\in B_{2r}} |u(z)|}{r^{k+1}} = 2^{k+1} \frac{\sup_{z\in B_{2r}} |u(z)|}{(2r)^{k+1}} \to 0
    \]
    as $r\to\infty$. Hence $D^\alpha u(x) = 0$ for any multi-index $\alpha$ of order $k+1$ and thus $D^\alpha u\equiv 0$, meaning that $u$ is a polynomial of degree at most $k$. 
    
    If $u$ is defined on the half-space, we see that its extension by odd reflection $u(\hat x, x_n) = - u(\hat x, -x_n)$ for $x_n$ is a harmonic function on the entire space due to the mean value characterization of harmonic functions. Thus, the same argument can be used.
\end{proof}

\section{\texorpdfstring{Barron spaces for ReLU$^\alpha$-activation}{Generalized Barron spaces}}\label{section barron background}

\subsection{Construction}
A neural network with a single hidden layer of width $n \in\mathbb N$ and activation function $\sigma:\R\to\R$ can be written as
\[
f_n(x) = \frac1n\sum_{i=1}^n a_i\,\sigma(w_i^Tx+b_i)
\]
where the parameters $a_i\in\R$, $w_i\in\R^d$ are the weights of the network and $b_i$ its biases. The normalizing factor $1/n$ is commonly included in theoretical analyses and ignored in practical implementations. Its presence is inconsequential, but allows for a simpler presentation -- see also \cite{representationformulas} for a discussion of equivalent parametrizations.

Barron spaces are function classes designed to include the infinite width limits of neural networks with a single hidden layer whose coefficients remain bounded. We explore these spaces mostly in the setting of activation functions which are powers of the popular ReLU activation. More general cases are considered in \cite{wojtowytsch2021kolmogorov, heeringa}.

Let $\pi$ be a probability measure on $\R\times\R^d\times \R$. We denote
\begin{equation}\label{eq f pi alpha}
f_{\pi,\alpha}:\R^d\to\R, \qquad f_{\pi,\alpha}(x) = \E_{(a,w,b)\sim\pi} \big[a\,\sigma_\alpha(w^Tx+b)\big]
\end{equation}
where
\[
\sigma_\alpha:\R\to\R, \qquad \sigma_\alpha(z) = \ReLU^\alpha(z) = \max\{z,0\}^\alpha.
\]
For a fixed point $x\in\R^d$, the function $(a,w,b)\mapsto a\,\sigma_\alpha(w^Tx+b)$ is continuous, hence Lebesgue measurable. It is Lebesgue integrable  for instance if
\[
 \E_{(a,w,b)\sim\pi} \big[|a|\,\sigma_\alpha(w^Tx+b)\big] \leq \E\big[ |a|\,\big|w^Tx+b\big|^\alpha] \leq \max\{1,|x|\big\}^\alpha\,\E\big[|a| \big(|w|+|b|\big)^\alpha] < \infty.
\]

\begin{lemma}
    Let $\pi$ be a probability measure on $\R\times\R^d\times \R$ such that
    \[
    \E\big[|a| \big(|w|+|b|\big)^\alpha] < \infty.
    \]
    Then the function $f_{\pi,\alpha}$ defined in \eqref{eq f pi alpha} is continuous.
\end{lemma}

\begin{proof}
    Let $x_n\to x$ in $\R^d$. Then $a\,\sigma_\alpha(w^Tx_n+b) \to a\,\sigma_\alpha(w^Tx+b)$ for all $(a,w,b)\in\R\times\R^d\times \R$. The sequence $x_n$ is bounded, so there exists $R>0$ such that
    \[
    \big|a\,\sigma_\alpha(w^Tx_n+b)\big|\leq |a|\,\big|w^Tx_n+b\big|^\alpha \leq |a|\big(|w|\,|x_n|+|b|\big)^\alpha \leq (1+R)^\alpha |a|\,\big(|w|+|b|\big)^\alpha
    \]
    for all $n$. By the dominated convergence theorem, we conclude that
    \[
    f_\pi(x_n) = \E_{(a,w,b)\sim\pi} \big[a\,\sigma(w^Tx_n+b)\big]
    \to  \E_{(a,w,b)\sim\pi} \big[a\,\sigma(w^Tx+b)\big] = f_\pi(x).\qedhere
    \]
\end{proof}

More quantitative statements are given in Section \ref{section embeddings}.

\begin{remark}
    With a little care, much of what we do also applies to the Heaviside activation function $\sigma_0(z) = 1_{\{z>0\}}$. Naturally, the represented functions are no longer continuous, but bounded and Borel-measurable. 
\end{remark}

The map $\pi \mapsto f_{\pi,\alpha}$ is not injective -- in fact, any measure $\pi$ which is invariant under the reflection map $(a,w,b)\mapsto(-a,w,b)$ represents the function $f_{\pi,\alpha}\equiv 0$. Further sources of non-uniqueness are explored in \cite{representationformulas}. We therefore make the following definition.

\begin{definition}[$\alpha$-Barron norm and space]
    Let $D\subseteq \R^d$ be a subset. For a continuous function $f:A\to\R$, we define
    \[
    \|f\|_{\B_\alpha(D)} = \inf \left\{\E_{(a,w,b)\sim\pi}\big[|a|\big(|w|+|b|\big)^\alpha\big] : f(x) = f_{\pi,\alpha}(x) \text{ for all } x\in D\right\}
    \]
    and
    \[
    \B_\alpha(D) = \{ f \in C^0(D) : \|f\|_{\B_\alpha(D)} < \infty\}.
    \]
\end{definition}

We call $\B_\alpha(D)$ the $\alpha$-Barron space over $D$. If $D=\R^d$ for some $d\in\N$, we simply denote $\B_\alpha(D) =: \B_{\alpha,d}$ or $\B_\alpha$ if $d$ is clear from context.

\begin{lemma}\cite[Lemma 1]{siegel2023characterization}
    Let $D\subseteq\R^d$. Then $\B_\alpha(D)$ is a Banach space.
\end{lemma}

Other classes of Barron spaces have been considered in the literature where only the magnitude of the `weight' variable $w$, but not the magnitude of the bias $b$ is controlled. Some technical difficulties arise, and subtly different geometric phenomena can be observed \cite{wojtowytsch2022optimal, boursier2023penalising}.

The function class that  we call `Barron space' here is also referred to as $\mathcal F_1$ or Radon-BV in the case $\alpha=1$, or, more technically, as the `variation space of the ReLU$^\alpha$-dictionary.' Sometimes, it is referred to as the `representational Barron space' to distinguish it from the `Barron class' or `spectral Barron space'. The ambiguity arises as the seminal article in the field \cite{barron1993} introduces versions of both function classes and proves the embedding of one into the other, giving rise to two notions of what `Barron space' should mean.

For us, Barron space will always be the {\em representational} Barron space.

\subsection{Function space embeddings}\label{section embeddings}
Let us observe some relationships between Barron spaces and more classical function classes.

\begin{lemma}\cite[Proposition 2.1]{heeringa}
    Let $R<\infty$ and $0<\beta<\alpha$. Assume that $D\subseteq \overline{B_R(0)}$ in $\R^d$. Then $\B_\alpha (D) \hookrightarrow \B_\beta(D)$ and
    \[
    \|f\|_{\B_\alpha(D)} \leq C_{\alpha,\beta, R}\,\|f\|_{\B_\beta(D)}.
    \]
\end{lemma}

Note that, globally, there is no embedding between $\B_\alpha$ and $\B_\beta$: While elements of $\B_\alpha$ are smoother at the origin for $\alpha>\beta$, they also grow faster at infinity in general. Thus, the restriction to a compact set is indeed necessary.

We also establish embeddings into classical function spaces: First for $\alpha\in\R\setminus \N$, then for $\alpha\in\N$.

\begin{theorem}\label{theorem fractional power embedding}
Assume that $\alpha = k+\gamma$ with $k\in \N_0$ and $\gamma\in(0,1)$. Assume that $\Omega\subseteq\R^d$ is a bounded domain.
\begin{enumerate}
\item $\B_\alpha$ embeds into $C^{k,\gamma}(\overline\Omega)$.
\item $\B_\alpha$ embeds into $W^{k+1,p}(\Omega)$ for all $p< (1-\gamma)^{-1}$.
\end{enumerate}
\end{theorem}

\begin{proof}
{\bf H\"older bound.} Note that for $0< x<z$ we have
\[
\sigma_\gamma(z) - \sigma_\gamma(x) = \int_x^z \frac{d}{ds}\sigma_\gamma(s)\ds = \gamma\int_x^z s^{\gamma-1}\ds \leq \gamma \int_{0}^{z-x}s^{\gamma-1}\ds = (z-x)^\gamma
\]
since $\sigma_\gamma'$ is monotone decreasing on $(0,\infty)$. Considering negative numbers separately, we find that $|\sigma_\gamma(x) - \sigma_\gamma(z)|\leq |x-z|^\gamma$ for all $x,z\in\R$. Hence, if $\beta =  (\beta_1, \dots, \beta_d)$ is a multi-index such that $|\beta| = \sum_{i=1}^d |\beta_i| = k$, then
\begin{align*}
    \partial^\beta \,\E_{(a,w,b)\sim\pi}& \big[a\,\sigma_\alpha(w^tx+b)\big] - \partial^\beta \,\E_{(a,w,b)\sim\pi} \big[a\,\sigma_\alpha(w^tz+b)\big]\\
    &= \E_{(a,w,b)\sim\pi} \big[a\,\partial^\beta\big\{\sigma_\alpha(w^tx+b) - \sigma_\alpha(w^Tz+b)\big\}\big]\\
    &= \prod_{i=1}^k (i+\gamma)\E_{(a,w,b)\sim\pi}\big[ a\,w_1^{\beta_1}\dots w_d^{\beta_d}\big\{\gamma\alpha(w^Tx+b) - \sigma_\gamma(w^Tz+b)\big\}\big]\\
    &\leq \prod_{i=1}^k (i+\gamma) \E_{(a,w,b)}\big[|a|\,|w|^k \big|w^Tx+b- w^Tz-b\big|^\gamma\big]\\
    &\leq \prod_{i=1}^k (i+\gamma) \E_{(a,w,b)}\big[|a|\,|w|^{k+\gamma}\big]\,|z-x|^\gamma
\end{align*}
where we can exchange the order of integration by applying the dominated convergence theorem to the difference quotients. Taking the infimum over all distributions $\pi$ representing a function $f\in\B_\alpha$, we find that the $k$-th derivatives of $f$ are $\gamma = \alpha-k$-H\"older continuous with norm at most $C \|f\|_{\B_k}$ for a constant $C$ depending only on $k$.

{\bf Sobolev bound.} Note again that
\[
\frac{d^{k+1}}{ds^{k+1}} = \prod_{i=0}^{k+1}(k+\gamma-i)\,\sigma_{\gamma-1}(s). 
\]
We choose $R>0$ such that $\Omega\subseteq B_R(0)$ and for a multi-index $\beta$ with $|\beta| = k+1$ we bound
\begin{align*}
\int_\Omega |\partial^\beta \sigma_\alpha(w^Tx+b)|^p \dx &\leq C\,|w|^{(k+1)p} \int_{B_R} \sigma_{\gamma-1} (w^Tx+b)^p\dx \leq C|w|^{(k+1)p} \int_{0}^{R} |w|^{(\gamma-1)p} |s|^{(\gamma-1)p}\ds \\
    &\leq C|w|^{(k+\gamma)p} C(p,R)
\end{align*}
for a constant $C(p,R)$ depending only on $p$ and $R$ which is finite if and only if $(\gamma-1)p>-1$, i.e.\ if and only if $p< (1-\gamma)^{-1}$.
\end{proof}

The integer case is slightly different. 

\begin{theorem}\label{theorem integer power embedding}
Assume that $\alpha = k\in\N$. Then $B_\alpha$ embeds continuously into $C^{k-1,1} = W^{k,\infty}$, but not into $C^k$.
\end{theorem}

\begin{proof}
The proof that $\B_k$ embeds into $C^{k-1,1}$ follows exactly as above. Even the activation function $\sigma_k$ is not $C^k$-smooth, so $\B_k$ cannot embed into $C^k$.
\end{proof}

There also exist embeddings from classical function spaces into Barron space. 

\begin{theorem}\label{theorem fractional sobolev embedding}
Let $f : \R^d\to\R$, $R>0$ and $t>k + d/2 + 1$. Then there exists $C$ depending on $k, d$ and $R$ and $F\in \B_k(\R^d)$ such that
\[
F\equiv f\text{ on }B_R(0)\qquad\text{and}\quad \|F\|_{\B_k} \leq C\,\|f\|_{H^t(\R^d)}
\]
where $H^t$ is the fractional Sobolev space with integrability $p=2$ and differentiability $t\in\R$.
\end{theorem}

\begin{proof}
By \cite[Proposition 4.1]{heeringa} (see also \cite[Propositions 7.1 and 7.4]{caragea2023neural} for optimality and \cite{wu2023embedding}), the bound
\[
\|f\|_\B \lesssim \int_{\R^d} \big(1+|\xi|^2\big)^{(k+1)/2}\,|\hat f(\xi)|\d\xi
\]
holds with a constant that may depend on $d, k$, but not on the function $f:\R^d\to\R$. Here $\hat f$ denotes the Fourier transform of $f$. Following the outline of \cite[Section IX]{barron1993}, we estimate by H\"older's inequality
\begin{align*}
\|f\|_\B &\lesssim \int_{\R^d} \big(1+|\xi|^2\big)^{-s} \big(1+|\xi|^2\big)^{(k+1)/2+s}\,|\hat f(\xi)|\d\xi\\
    &\leq \left(\int_{\R^d} \big(1+|\xi|^2\big)^{k+1+2s}\,|\hat f(\xi)|^2\d\xi\right)^\frac12 
    \left(\int_{\R^d} \big(1+|\xi|^2\big)^{-2s}\,|\hat f(\xi)|\d\xi\right)^\frac12.
\end{align*}
The second factor is finite if and only if $4s>d$ and the first is the norm (or rather, one of many equivalent versions of the norm, see \cite[Section 6.10]{dobrowolski2010angewandte}) in the fractional Sobolev space of order $t = k+1+2s > k + 1 + d/2$. 
\end{proof}

\subsection{Approximation by finite neural networks in function space norms}

Barron spaces are popular in the context of machine learning because they provide model classes of functions which can be approximated well by neural networks with a single hidden layer and `small' coefficients, i.e.\ we do not have to rely on miraculous cancellations. This is indicative of the fact that gradient-based optimizers can efficiently `learn' a solution, although no rigorous link has been established. We present a representative result.

While the techniques of proof are well-known in the area going back to \cite{barron1993}, we have been unable to locate a reference for the precise version and sketch the proof for the reader's convenience.

\begin{theorem}
Let $\Omega\subseteq\R^d$ be a bounded open set and $m\in\N$ and $q\in[2,\infty)$. Let $\alpha = k+\gamma$ for $\gamma\in(0,1]$ and assume that
\begin{enumerate}
    \item $m\leq k$ or
    \item $m = k+1$ and $(1-\gamma)q < 1$
\end{enumerate}
Then for every $f\in \B_\alpha(\Omega)$ and every $n\in\N$, $\eps>0$, there exists a finite $\sigma_\alpha$-network
\[
f_n(x) = \frac1n\sum_{i=1}^n a_i\,\sigma_\alpha(w_i^Tx+b_i)
\]
such that
\begin{enumerate}
    \item $\frac1n\sum_{i=1}^n|a_i|\,\big(|w_i|+|b_i|\big)^\alpha \leq \|f\|_{\B_\alpha}$.

    \item There exists a constant $C>0$ depending only on $m,q$ and $\Omega$ such that
    \[
    \|f-f_n\|_{W^{m,q}(\Omega)} \leq C\,\|f\|_{\B_\alpha}\,n^{-1/2}.
    \]
\end{enumerate}
\end{theorem}

\begin{proof}
    We note that the unit ball $\B_\alpha$ can be characterized as the closed convex hull of 
    \[
    \mathbb B := \{x\mapsto a\,\sigma_\alpha(w^Tx+b) : a=\pm 1, |w|+|b|\leq 1\}
    \]
    where the closure can be taken in any space into which $\B_\alpha$ embeds by Theorems \ref{theorem fractional power embedding} and \ref{theorem integer power embedding}. The reason for the reduction is the homogeneity of the activation function $\sigma_\alpha$, which allows us to consider only the unit ball of $\R^{d+1}$. Details of the argument can be found in \cite[Section 2]{representationformulas} for the case $k=1$, but apply more generally. In particular, there exists $M\in\N$ such that

    Recall that $L^p(\Omega)$ is a type II Banach space for $p\in[2,\infty)$, i.e.\ for every set of functions $f_1,\dots, f_n\in L^p(\Omega)$, there exists $C>0$ such that
    \[
    \E_\xi \left[\left\|\sum_{i=1}^n \xi_if_i\right\|_{L^p(\Omega)}\right] \leq C \sum_{i=1}^n \|f_i\|_{L^p(\Omega)}^2
    \]
    where the random variables $\xi_i$ are iid and take values $\pm 1$ with equal probability $1/2$. The same is true for $W^{k,p}$. We write $Z$ below for any type II Banach space which $\B_\alpha$ embeds into. The result now follows from Theorems \ref{theorem fractional power embedding} and \ref{theorem integer power embedding}.
    
    Let $f = \E_{(a,w,b)\sim\pi} [a\,\sigma(w^T\cdot + b)]$ for a probability distribution on the unit ball of $\R\times \R^d\times \R$ and assume that $(a_i, w_i, b_i)_{i=1}^m$ are iid samples from $\pi$. Then we `symmetrize', i.e.\ we note that we can introduce $m$ additional iid samples $(a_i', w_i', b_i')$ from $\pi$ and get
    \[
    f(x) = \E_{(a_i', w_i',b_i')}\left[\frac1m\sum_{i=1}^m a_i'\,\sigma(w_i'^Tx+b_i')\right].
    \]
    This yields
    \begin{align*}
    \E_{(a_i,w_i,b_i)_{i=1}^m\sim \pi^m} &\Bigg[\bigg\|f - \frac1m\sum_{i=1}^ma_i\,\sigma(w_i\cdot + b_i)\bigg\|_{Z}^2 \Bigg]\\
        &\leq \E_{(a_i,w_i,b_i)_{i=1}^m\sim \pi^m} \Bigg[\bigg\|\E_{(a_i', w_i',b_i')}\left[\frac1m\sum_{i=1}^m a_i'\,\sigma(w_i'^Tx+b_i')\right] - \frac1m\sum_{i=1}^ma_i\,\sigma(w_i\cdot + b_i)\bigg\|_{Z}^2 \Bigg]\\
        &= \E_{(a_i,w_i,b_i)_{i=1}^m\sim \pi^m} \Bigg[\bigg\|\E_{(a_i', w_i',b_i')}\frac1m\sum_{i=1}^m\left[a_i'\,\sigma(w_i'^Tx+b_i') - a_i\,\sigma(w_i\cdot + b_i)\right]\bigg\|_{Z}^2 \Bigg]\\
        &\leq \frac1{m^2}\E_{(a_i,w_i, b_i), (a_i', w_i', b_i')\sim\pi^{2m}} \left[  \bigg\|\sum_{i=1}^m \{a_i'\,\sigma(w_i'^T\cdot +b_i) - a_i\,\sigma(w_i\cdot +b_i)\}\bigg\|^2_X\right]\\
        &= \frac1{m^2}\E_{(a_i,w_i, b_i), (a_i', w_i', b_i')\sim\pi^{2m}, \xi } \left[  \bigg\|\sum_{i=1}^m \xi_i\{a_i'\,\sigma(w_i'^T\cdot +b_i) - a_i\,\sigma(w_i\cdot +b_i)\}\bigg\|^2_X\right]
    \end{align*}
    where we introduced an additional independent random quantity $\xi_i$ as above. This was permissible since it does not change the distribution of $a_i'\,\sigma(w_i'^T\cdot +b_i) - a_i\,\sigma(w_i\cdot +b_i)$ since both copies of the parameters $(a,w,b)$ are indistinguishable. By the type II property of $Z$, we find that
\begin{align*}
\E_{(a_i,w_i, b_i), (a_i', w_i', b_i')\sim\pi^{2m}, \xi }& \left[  \bigg\|\sum_{i=1}^m \xi_i\{a_i'\,\sigma(w_i'^T\cdot +b_i) - a_i\,\sigma(w_i\cdot +b_i)\}\bigg\|^2_Z\right]\\
    &\leq C_X \,\E_{(a_i,w_i, b_i), (a_i', w_i', b_i')\sim\pi^{2m}} \left[\sum_{i=1}^m \| a_i'\,\sigma(w_i'^T\cdot +b_i) - a_i\,\sigma(w_i\cdot +b_i)\|^2_Z\right]\\
    &= 4C_X\,m\,\E\big[\|a\,\sigma(w^T\cdot+b)\|_X^2\big]\\
    &\leq C\,\|f\|_{\B_\alpha}^2m
\end{align*}
for a constant $C$ which combines the type II constant of $Z$ and the embedding constant of $\B_k$ into $Z$. Thus, overall:
\[
\E_{(a_i,w_i,b_i)_{i=1}^m\sim \pi^m} \Bigg[\bigg\|f - \frac1m\sum_{i=1}^ma_i\,\sigma(w_i\cdot + b_i)\bigg\|_{Z}^2 \Bigg]\leq \frac{C\,\|f\|_{\B_k}^2}m.
\]
In particular, there exist weights $(a_i, w_i, b_i)_{i=1}^m$ for which the Monte-Carlo rate is attained.
\end{proof}

With different methods, approximation rates can also be established in $C^0$, which is not a type-II Banach space -- see e.g.\ \cite[Theorem 12]{ma2020towards} for the case $k=1$. We conjecture that a similar analysis may apply more to approximation in $C^{k-1}$ for $k\in\mathbb N$. Notably, we generally cannot expect approximation in the (non-separable) H\"older spaces $C^{0,\alpha}$. 

If $q\leq 2$, $W^{k,q}$ is a Banach space of type $q$, not type 2. The same approximation rates can of course be established in the weaker topology by the embedding from $W^{k,2}$ if $m\leq k$ or $2(1-\gamma)<1$, but to the best of our knowledge it is not clear whether the rate of $n^{-1/2}$ can be attained if $m = k+1$ and $q$ is such that $(1-\gamma)q<1$ but $(1-\gamma)2\geq 1$.

Similar results with sharper asymptotic rates in the `spectral Barron class' can be found in \cite[Theorem 1]{siegel2022high}. The spectral Barron class is defined using the Fourier transform, similar to the argument in Theorem \ref{theorem fractional sobolev embedding}.

\subsection{One dimensional Barron spaces and dimension reduction}
We consider Barron spaces in one dimension in greater detail, as we will need them below to construct specific Barron functions also in higher dimension.

\begin{lemma}\label{lemma one dimension derivatives}
    Assume that $\alpha = k\in\N$. Then $f\in \B_\alpha(-1,1)$ if and only if the $k+1$-th derivative of $f$ satisfies
    \[
    \int_\R \big|f^{(k+1)}(x)\big|\,\big(1+|x|^k\big)\dx < +\infty.
    \]
\end{lemma}

The statement is a slightly different statement of \cite[Proposition 2.3]{heeringa}, where the authors provide a much more general analysis also for non-integer degree of smoothness using Caputo-derivatives, but only on bounded intervals. In the case $k=1$, a proof is given for instance in \cite[Example 4.1]{representationformulas} -- see also \cite{li2020complexity}. For the reader's convenience, we give a full proof of this version in the appendix.

We can use Lemma \ref{lemma one dimension derivatives} to prove a result which will be important later.

\begin{corollary}\label{corollary non-barron example}
    Let $f(x) = x^k \log x$. Then $f\notin \mathcal B_k([0,1])$, i.e.\ there exists no $\mathrm{ReLU}^k$-Barron function $F$ on $\R$ such that $F(x) = f(x)$ for all $x\in[0,1]$.
\end{corollary}

\begin{proof}
    We compute
    \begin{align*}
    \frac{d^{k+1}}{dx^{k+1}}f(x) &= \sum_{l=0}^{k+1}\binom{k+1}l \,\frac{d^l}{dx^l}x^k\,\frac{d^{k+1-l}}{dx^{k+1-l}}\log x\\
        &= \sum_{l=0}^{k} \binom{k+1}l \,\frac{k!}{(k-l)!}x^{k-l}\,(-1)^{k-l}(k-l)!\,x^{-(k+1-l)}\\
        &= \left(\sum_{l=0}^k (-1)^{k-l} \,\binom{k+l}l \,k!\right)\frac1x
    \end{align*}
    since in the term for $l=k+1$ where no derivatives fall on the logarithm, all $k+1$ derivatives fall on the factor $x^k$, making it zero. The constant factor is non-zero since $f$ is not a polynomial of degree $\leq k$ on $[0,1]$, i.e.\ its $k+1$-th derivative cannot vanish. In particular, like the function $1/x$, the $k+1$-th derivative of $f$ is not integrable over $(0,1)$.
\end{proof}

One-dimensional Barron functions play an important role due to the following dimension reduction techniques.

\begin{lemma} \label{lemma dimension reduction 1}
    Let $C\subset\R^d$ be a set which contains a line segment $L= \{x_0+tv : t\in (0,\eps)\}$. If $u\in \B_\alpha(C)$, then the `sliced' function $u_L:(0,\eps)\to \R$ given by $u_L(t) = u(x+tv)$ satisfies $u_L\in \B_\alpha(0,\eps)$.
\end{lemma}

\begin{proof}
    Since $u\in\B_\alpha$, we find that
    \[
    u(x) = \E_{(a,w,b)\sim\pi}\big[a\,\sigma_\alpha(w^Tx+b)\big]
    \]
    for all $x\in C$ and in particular
    \[
    u_L(t) = u(x_0+tv) = \E_{(a,w,b)\sim\pi}\big[a\,\sigma_\alpha\big(t\,w^Tv + w^Tx_0+b\big)\big].
    \]
    Since the moment-bound
    \[
    \E_{(a,w,b)\sim\pi}\big[|a|\big\{|w^Tv| + |w^Tx_0+b|\big\}^\alpha\big] \leq 2^\alpha \max\{1, |x_0|, |v|\}^\alpha \,\E\big[ |a| \big\{|v| + |b|\big\}^\alpha\big] 
    \]
    holds, we find that
    \[
    \|u_L\|_{\B(0,\eps)} \leq  2^\alpha \max\{1, |x_0|, |v|\}^\alpha\,\|u\|_{\B(C)}. \qedhere
    \]
\end{proof}

A converse of the preceding statement is true for {\em homogeneous} Barron functions on a half-plane.

\begin{lemma} \label{lemma dimension reduction 2}
    Let $\R^2_+=\{(x,y)\in \R^2 : y>0\}$ and $u :\R^2_+\to\R$ such that $u(\lambda x) = \lambda^\alpha u(x)$ for all $\lambda>0$ and $x\in \R^2_+$. Then $u\in \B_\alpha(\R^2_+)$ if and only if $u(\cdot, 1)\in \B(\R)$.
\end{lemma}

\begin{proof}
    In light of Lemma \ref{lemma dimension reduction 1}, we only have to establish the second direction.
    Assume that $u(\cdot, 1)\in \B_\alpha(\R)$, i.e.\
    \[
    u(x,1) = \E_{(w, b)\sim\pi}\big[a\,\sigma_\alpha(wx+b)\big].
    \]
    Then
    \[
    u(x,y) = y^\alpha u(x/y, 1) = y^\alpha \E_{(w, b)\sim\pi}\big[a\,\sigma_\alpha(wx/y+b)\big] = \E_{(w, b)\sim\pi}\big[a\,\sigma_\alpha(wx+by)\big]
    \]
    by positive homogeneity and the fact that $y>0$. In particular, $u\in \B_\alpha(\R^2_+)$.
\end{proof}

\section{\texorpdfstring{PDEs in ReLU$^\alpha$-Barron spaces on a half-space}{PDEs in generalized Barron-spaces on half-space}}
\label{section pde}

\subsection{Explicit representation formulas}\label{section explicit}

In the plane $\R^2$, harmonic functions (two real variables) and holomorphic functions (one complex variable) are intimately linked, which we will exploit below. Recall that the $\alpha$-th power of a complex number $z = x+iy = r\,e^{i\phi}$ is defined uniquely if and only if $\alpha$ is an integer. If $\alpha\notin\N$, then we can define
\[
z^\alpha= \big(re^{i\phi}\big)^\alpha = r^\alpha\,e^{i\alpha\phi}
\]
only non-uniquely since the phase $\phi$ is only defined as an element of $\R/2\pi\Z$ and since
\[
e^{i\alpha\phi} = e^{i\alpha (\phi+2\pi n)}\qquad\LRa\quad e^{2\pi i \:n \alpha} = 1 \qquad\LRa\quad n\alpha\in\mathbb Z.
\]
If $\alpha$ is an integer, the definition of the power does not depend on the choice of representative for the angle $\phi$ (i.e.\ on the choice of $n\in\Z$. If $\alpha$ is a rational number, there are a finite number of different $\alpha$-th powers of $z$. If $\alpha$ is an irrational number, there are a countably infinite number of $\alpha$-th powers of $z$. In the following, we choose the $\alpha$-th power version of $z$ which corresponds to the selection of representative $\phi\in[0,2\pi)$ for $z = re^{i\phi}$ everywhere (the `principal branch').

\begin{theorem}\label{theorem complex representation}
Assume that 
\begin{equation}\label{eq harmonic dirichlet alpha}
\begin{pde}
-\Delta u &= 0 &\text{in }\R^2_+\\
u &= \ReLU^\alpha(x) &\text{on }\partial\R^2_+.
\end{pde}
\end{equation}
Then $u = u_\alpha + h$ where
\begin{enumerate}
\item $u_\alpha$ is given by
\[
    u_\alpha(x,y) = 
    \begin{cases} \left(\frac1\pi\,\arctan\left(\frac xy\right) + \frac12\right)\Re \big((x+iy)^\alpha\big) - \frac{1}{\pi}\log\big(\sqrt{x^2+y^2}\big)\,\Im\big((x+iy)^\alpha\big)  & \text{if }\alpha \in\N\vspace{3mm}\\ 
    \Re{((x+iy)^{\alpha})} + \cot{(\pi\alpha)}\,\Im{((x+iy)^{\alpha})} &\text{if }\alpha \in (0,\infty)\setminus \N.
    \end{cases}
\]
\item $h$ is a harmonic function on $\R^2_+$ and $h\equiv 0$ on $\partial\R^2_+= \R\times\{0\}$.
\end{enumerate}
Here $\Re, \Im$ denote the real and imaginary parts of a complex number respectively.
\end{theorem}

\begin{proof}
\textbf{Case 1}: We first consider the case $\alpha \in \mathbb{N}$.
Let us begin by verifying the harmonicity of $u_{\alpha}$ as specified above. To simplify the expressions below, we write 
\[
r(x,y) := \Re\big((x+iy)^\alpha), \qquad i (x,y) =\Im\big((x+iy)^\alpha\big).
\]
Recall the Cauchy-Riemann equations
\[
\partial_x r = \partial_y i, \qquad \partial_y r = -\partial_x i.
\]
For $\alpha\in\N$, these hold automatically since $z\mapsto z^\alpha$ is uniquely defined and holomorphic.
We compute 
\begin{align*}
\partial_{x}u_{\alpha}(x,y) 
    &= \frac{y}{\pi(x^{2}+y^{2})}r(x,y)+ \left(\frac{1}{2} + \frac{1}{\pi}\arctan{\left(\frac{x}{y}\right)}\right)\partial_{x}r(x,y)\\
    &\qquad - \frac{x}{\pi(x^{2}+y^{2})}i(x,y)- \frac{1}{\pi}\log{(\sqrt{x^{2}+y^{2}})}\partial_{x}i(x,y)\\
\partial_{y}u_{\alpha}(x,y)
    &= -\frac{x}{\pi(x^{2}+y^{2})}r(x,y)+ \left(\frac{1}{2} + \frac{1}{\pi}\arctan{\left(\frac{x}{y}\right)}\right)\partial_{y}r(x,y)\\
    &\qquad -\frac{y}{\pi(x^{2}+y^{2})}i(x,y)- \frac{1}{\pi}\log{\left(\sqrt{x^{2}+y^{2}}\right)}\partial_{y}i(x,y)\\
\partial_{xx}u_{\alpha}(x,y)
    &= -\frac{2xy}{\pi(x^{2}+y^{2})^{2}}r(x,y)+ \frac{2y}{\pi(x^{2}+y^{2})}\partial_{x}r(x,y)+ \left(\frac{1}{2}+ \frac{1}{\pi}\arctan{\left(\frac{x}{y}\right)}\right)\partial_{xx}r(x,y)\\
    &\qquad -\frac{y^{2}-x^{2}}{\pi(x^{2}+y^{2})^{2}}i(x,y)-\frac{2x}{\pi(x^{2}+y^{2})}\partial_{x}i(x,y)- \frac{1}{\pi}\log{(\sqrt{x^{2}+y^{2}})}\partial_{xx}i(x,y)\\
\partial_{yy}u_{\alpha}(x,y) 
    &= \frac{2xy}{\pi(x^{2}+y^{2})^{2}}r(x,y)-\frac{2x}{\pi(x^{2}+y^{2})}\partial_{y}r(x,y)+ \left(\frac{1}{2} + \frac{1}{\pi}\arctan{\left(\frac{x}{y}\right)}\right)\partial_{yy}r(x,y)\\
    &\qquad -\frac{x^{2}-y^{2}}{\pi(x^{2}+y^{2})^{2}}i(x,y)-\frac{2y}{\pi(x^{2}+y^{2})}\partial_{y}i(x,y)- \frac{1}{\pi}\log{(\sqrt{x^{2} +y^{2}})}\partial_{yy}i(x,y).
\end{align*}
Hence:
\begin{align*}
\Delta u_{\alpha}(x,y) 
    &= \left(\frac{1}{2} + \frac{1}{\pi}\arctan{\left(\frac{x}{y}\right)}\right)\,\Delta r(x,y)- \frac{1}{\pi}\log{\left(\sqrt{x^{2}+y^{2}}\right)}\,\Delta i(x,y)\\
    &\qquad + \frac{2y}{\pi(x^2+y^2)}\big(\partial_xr - \partial_yi\big) - \frac{2x}{\pi(x^2+y^2)}\big(\partial_yr + \partial_xi\big) \hspace{2cm} = 0
\end{align*}
It is easy to see  that $u_{\alpha}(x,y)$ satisfies the boundary condition since
\[
\lim_{y\to 0^{+}}u_{\alpha}(x,y) = \left(\frac{1}{2} + \frac{1}{\pi}\frac{\pi}{2}\text{sign}(x)\right)x^{\alpha} - \frac{1}{\pi}\log(x)\cdot 0 = \frac{x^{\alpha} + x^{\alpha}\text{sign}(x)}{2} = \text{ReLU}^{\alpha}(x)
\]
Thus, $u_\alpha$ is a solution of the PDE \eqref{eq harmonic dirichlet alpha}. Every other solution $\tilde u_\alpha$ satisfies $\Delta (\tilde u_\alpha - u_\alpha)= 0$ in $\R^2_+$ and $u-u_\alpha\equiv 0$ on $\partial \R^2_+$ as we claimed.

\textbf{Case 2}: We now consider the case that $\alpha \in \mathbb{R}\backslash\mathbb{N}$.
It can be easily shown that $u_{\alpha}(x,y)$ as defined in the theorem statement is harmonic:
\[
\Delta u_{\alpha}(x,y) = \Delta\,\Re{((x+iy)^{\alpha})} - \cot{(\pi\alpha)}\,\Delta\,\Im{((x+iy)^{\alpha})} = 0
\]
since the real and imaginary part of a holomorphic function are harmonic. To see that $z^\alpha$ is indeed a holomorphic function on $\R^2_+$ note that we can define $z^\alpha = \exp(\alpha\,\log(z))$ locally with any branch of the (complex) logarithm, see e.g.\ \cite[Theorem I.2.13]{busam2009complex}. Since the composition of holomorphic functions is holomorphic, the function $z\mapsto z^\alpha$ is holomorphic on any domain where it is uniquely defined, e.g.\ $\R^2_+\subseteq\C$.

Let us demonstrate that $u_\alpha$ also satisfies the desired boundary condition:
\[
(x+iy)^{\alpha} = r^{\alpha}e^{i\alpha\theta} = (x^{2}+y^{2})^{\alpha/2}(\cos{(\alpha\arctan(y/x))} + i\sin{(\alpha\arctan(y/x))})
\]
On the positive $x$-axis, we selected the representative of $z$ for which $\phi = 0$, which yields $\phi = \pi$ on the negative $x$-axis, if the argument $\phi$ is selected continuously in $\R^2_+$.
If $x>0$, we have
\[
    \lim_{y\to 0^{+}} u_{\alpha}(x,y) = (x^{2}+0)^{\alpha/2}(\cos{(\alpha \cdot 0)} -\cot{(\alpha\pi)}\sin{(\alpha\cdot 0)}) = x^{\alpha}
\]
and if $x<0$:
\[
    \lim_{y\to 0^{+}}u_{\alpha}(x,y) = (x^{2} + 0)^{\alpha/2}(\cos{(\alpha\pi)} - \cot{(\alpha\pi)}\sin{(\alpha\pi)}) = 0
\]
Thus, the harmonicity and boundary conditions are verified.
\end{proof}

The explicit expression is particularly meaningful in the case of fractional $\alpha$.

\begin{lemma}\label{lemma unique homogeneous solution}
    Let $\alpha \in(0,\infty)\setminus\N$. Then $u_\alpha$ is the unique harmonic function such that $u(\lambda x, \lambda y) = \lambda^\alpha\,u(x,y)$ and for all $\lambda>0$ and $u(1,0)=1$, $u(-1,0) = 0$.

    For $\alpha\in\mathbb N$, no homogeneous solution exists.
\end{lemma}

\begin{proof}
    Together, the conditions imply that 
    \[
    u(x,0) = |x|^\alpha \,u(x/|x|, 0) = |x|^\alpha \,\ReLU^0(x) =  \ReLU^\alpha(x).
    \]
    Assume that $U, V$ are two functions both satisfying the PDE $\Delta u \equiv 0$ and the boundary condition $u(\lambda x) = \ReLU^\alpha(x)$. Then in particular 
    \[
    \begin{pde}
        \Delta(U-V) &= 0 &\text{in }\R^2_+\\
        U-V &= 0 &\text{on }\partial \R^2_+
    \end{pde}
    \]
    and
    \[
    \max_{x\in B_r(0)} |U(x)-V(x)| = r^\alpha \max_{x\in B_1(0)}|U(x)-V(x)| = C\,r^\alpha
    \]
    for all $r>0$. By Liouville's Theorem (Theorem \ref{theorem liouville}), we find that $U-V$ is a harmonic polynomial of degree $\leq \lfloor \alpha\rfloor$ where $\lfloor\cdot\rfloor$ denotes the smallest integer below $\alpha$. Additionally, we conclude that $U-V$ is positively homogeneous of degree $\alpha$.

    Since there is no non-zero polynomial which is homogeneous of degree $\alpha\notin\N$, we find that $U-V\equiv 0$. A similar argument applies in the integer case.
\end{proof}

In general, translating the expressions for $u_\alpha$ in Theorem \ref{theorem complex representation} into Cartesian coordinates using
\begin{align*}
r^\alpha e^{i\alpha\phi} &= (x^2+y^2)^\frac\alpha2 \big(\cos (\alpha\,\cot^{-1}(x/y)) +i\, \sin(\alpha\,\cot^{-1}(x/y))\big)
\end{align*}
does not yield a simple and satisfying expression. We give some examples where it does.

\begin{example}[$\alpha=1,2,3$]
    For integer values of $\alpha$, we can easily give explicit expressions from the binomial formula, for instance:
    % \todo[inline, color =  blue!30]{
    % Could you fill this in (and double-check the calculations above and below for this section)?
    % }
    \begin{align*}
        u_{1}(x,y) &= \left(\frac{1}{\pi}\arctan{\left(\frac{x}{y}\right)}+\frac{1}{2}\right)\,x \:-\: \frac{1}{\pi}\log\big(\sqrt{x^2 + y^2}\big)\,y\\
        u_{2}(x,y) &= \left(\frac{1}{\pi}\arctan{\left(\frac{x}{y}\right)}+\frac{1}{2}\right) \big(x^2 - y^2\big) \:-\: \frac{1}{\pi}\log\left(\sqrt{x^2 + y^2}\right) \,2xy\\
        u_{3}(x,y) &= \left(\frac{1}{\pi}\arctan{\left(\frac{x}{y}\right)}+\frac{1}{2}\right) \big(x^3 - 3xy^2\big) \:-\: \frac{1}{\pi}\log\big(\sqrt{x^2 + y^2}\big)\: \big(3x^2y -y^3 \big).
    \end{align*}
\end{example}

\begin{example}[$\alpha = 1/2$]
For $\alpha=1/2$, we can find an explicit formula. This case is particularly convenient since $\cot(\pi/2) = 0$, so
\[
u_{1/2}(x,y) = \Re\big((x+iy)^{1/2}\big) = r^{1/2} \cos (\phi/2)
\]
in polar coordinates. Recall the half-angle formula
\[
\cos(\phi/2) = \sqrt{\frac{1+\cos (\phi)}2}.
\]
Respecting the homogeneity of the function, we make the ansatz
\[
    u_{1/2}(x,y) = r^{1/2}\sqrt{\frac{1+\cos{(\phi)}}{2}}=\sqrt{\sqrt{x^{2}+y^{2}}\left(\frac{1 + (x/\sqrt{x^{2}+y^{2}})}{2}\right)} = \sqrt{\frac{\sqrt{x^2+y^2}+x}2}.
\]
It is easy to verify that $u_{1/2}$ as defined is harmonic since 
\begin{align*}
    \partial_{x}u_{1/2}(x,y) &= \frac{\sqrt{\sqrt{x^{2}+y^{2}} + x }}{2\sqrt{2}\sqrt{x^{2}+y^{2}}}\\
\partial_{xx}u_{1/2}(x,y) &= \frac{(\sqrt{x^{2}+y^{2}}-2x)(\sqrt{\sqrt{x^{2}+y^{2}}+x})}{4\sqrt{2}(x^{2}+y^{2})^{3/2}}\\
\partial_{y}u_{1/2}(x,y) &= \frac{y}{2\sqrt{2}\sqrt{x^{2}+y^{2}}\sqrt{\sqrt{x^{2}+y^{2}}+x}}\\
\partial_{yy}u_{1/2}(x,y) &= - \frac{(\sqrt{x^{2}+y^{2}}-2x)(\sqrt{\sqrt{x^{2}+y^{2}}+x})}{4\sqrt{2}(x^{2}+y^{2})^{3/2}}
\end{align*}
and thus $\Delta u_{1/2}\equiv 0$. The boundary condition is satisfied since
\[
    \lim_{y\to 0^{+}} u_{1/2}(x,y) = \sqrt{\frac{\sqrt{x^{2}+0}+x}{2}} = \sqrt{\frac{|x|+x}{2}} = \sqrt{\ReLU(x)}.
\]
It is easy to see that the given expression for $u_{1/2}$ satisfies $u(\lambda x) = \lambda^{1/2}u(x)$ for $\lambda\geq 0$, meaning that we have identified the (unique) $1/2$-homogeneous solution of the equation.
\end{example}

\begin{example}
For $\alpha=3/2$, similar to $\alpha=1/2$ we have $\cot{(3\pi/2)}=0$, so
\[u_{3/2}(x,y) = \Re{((x+iy)^{3/2})} = r^{3/2}\cos{(3\phi/2)}\]
in polar coordinates. Recall the triple-angle formula
\[
\cos (3\theta) =\cos^3\theta - 3\sin^2\theta\,\cos\theta.
\]
which indicates that 
\[
u_{3/2}(x,y) = r^{3/2}\left(\frac{1+\cos{(\phi)}}{2}\right)^{3/2} - 3r^{3/2}\left(\frac{1-\cos{(\phi)}}{2}\right)\sqrt{\frac{1+\cos{(\phi)}}{2}}
\]
and in Cartesian variables
\[
    u_{3/2}(x,y) = \left(\frac{\sqrt{x^{2}+y^{2}}+x}{2}\right)^{3/2} - 3 \left(\frac{\sqrt{x^{2}+y^{2}}+x}{2}\right)^{1/2}\left(\frac{\sqrt{x^{2}+y^{2}}-x}{2}\right).
\]
Harmonicity and attainment of the boundary condition can be verified explicitly as above.
\end{example}

\subsection{\texorpdfstring{Small fractional powers: ReLU$^\alpha$ for $0<\alpha<1$}{Small fractional powers of ReLU activation}}\label{section fractional alpha}

We now show that for non-integer powers of ReLU activation $\sigma_\alpha$ with $\alpha\in(0,1)$, the harmonic function $u_\alpha$ is indeed an element of the $\sigma_\alpha$-Barron space. As a by-product of our proof, the same holds for bounded Lipschitz-continuous activation functions such as tanh or sigmoid.

\begin{lemma}\label{lemma other representation}
Let $g:\R\to\R$ be a continuous function such that $|g(t)| \leq C(1+|t|)^\alpha$ for some $C>0$ and $\alpha\in(0,1)$. There exists a unique solution of the boundary value problem
\[
\begin{pde}
-\Delta u &= 0 &\text{in }\{(x_1,x_2):x_2>0\} \\
u&= g(x_1) & x_2 =0\\
\frac{u(x)}{\|x\|} &\to 0  &\text{as }\|x\|\to \infty
\end{pde}
\]
and it is given by 
\begin{equation}\label{eq other representation formula}
u(x,y) = \frac1\pi \int_\R \frac{1}{1+t^2}\,g(x+ty)\dt.
\end{equation}
\end{lemma}
\begin{proof}
The fact that \eqref{eq other representation formula} defines a harmonic function with boundary values $\sigma$ is a straight-forward modification of the proof of \cite[Theorem 2.2.14]{evans} with the substitution
\[
u(x, y) = \frac1\pi\int_\R \frac{g(s)}{(x-s)^2 +y^2}\,y\ds = \frac1\pi\int_\R \frac{g\left(y\,\frac{s-x}y+x\right)}{\left(\frac{x-s}y\right)^2 +1}\,\frac1y\ds = \frac1\pi\int_\R \frac{g\left(yt+x\right)}{t^2 +1}\dt
\]
i.e.\ $t = (x-s)/y$.
The only difference is the growth condition at infinity: \cite{evans} assumes that $g$ is bounded, but the proof applies in our more general context.

We note that
\[
|u(x,y)| \leq \frac1\pi \int_\R \frac{C(1+ |x| + |t|\,|y|)^\alpha}{1+t^2} \dt \leq \frac{C\,(1+|x|+|y|)^\alpha}\pi\int_\R \frac{(1+|t|)^\alpha}{1+t^2}\dt \leq C' \big( 1+ \|(x,y)\|\big)^\alpha
\]
for some $C'>0$. In particular, the representation formula provides a harmonic function which grows slower than linearly at infinity. As in Lemma \ref{lemma unique homogeneous solution}, the uniqueness follows from Liouville's Theorem. Namely, if $v$ is another harmonic function on $\R^2_+$ with boundary values $\sigma$ , then $u-v$ is a harmonic function on $\R^2_+$ with boundary values $0$. If both $u$ and $v$ grow sublinearly at infinity, so does their difference, meaning that $u-v$ is a harmonic polynomial of degree zero (i.e.\ a constant function). Since $u= v$ on $\partial\R^2_+$, the constant must be zero.
\end{proof}

The conditions on $g$ could obviously be weakened, but we see little application for such extensions. Note that in this form, Lemma \ref{lemma other representation} applies not only to powers of the ReLU activation $\sigma_\alpha(x)$ with $\alpha<1$, but also to popular activation functions such as tanh or sigmoid.

\begin{corollary}
    Let $\alpha\in(0,1)$ and $g \in \B_\alpha(\R)$. Then there exists a unique solution $u_\alpha\in \B_\alpha(\R^2_+)$ of the PDE
    \begin{equation}\label{eq harmonic half-space alpha}
    \begin{pde}
        -\Delta u &= 0 & \text{in }\R^2_+\\
        u &= g &\text{on }\partial\R^2_+.
    \end{pde}
    \end{equation}
\end{corollary}

\begin{proof}
    Let $x\in\R$ and $y\geq 0$ be fixed and $\mu$ a probability measure on (the Borel $\sigma$-algebra of) $\R^3$ such that
    \[
    g(t) = \E_{(a,w,b)\sim\mu}\big[a\,\sigma_\alpha(w t+b)\big].
    \]
    Then the unique sublinearly growing solution of the PDE \eqref{eq harmonic half-space alpha} is
    \begin{align*}
    u_\alpha(x,y) &= \frac1\pi\int_\R \frac{g(x+ty)}{1+t^2}\dt\\
        &= \int_\R \left(\int_{\R^3}  a\,\sigma_\alpha(w(x+ty)+b) \d\mu_{(a,w,b)}\right)\, \frac{1}{\pi(1+t^2)}\dt\\
        &=\int_{\R\times \R^2\times \R} \sigma_\alpha(w_1x+w_2y+b)\d\mu^\sharp_{(a,w,b)}
    \end{align*}
    where the probability measure $\mu^\sharp$ is the push-forward of the product probability measure 
    \[
    \mu \otimes \left(\frac1{\pi(1+t^2)}\cdot \L^1\right)
    \]
    on $\R^3\times\R$ along the map $(a,w,b; t) \mapsto (a, w, tw, b)$. The identification can be made by Fubini's theorem since the integrand is continuous, hence jointly Borel-measurable and, due to the moment bound on $\mu$, integrable. We observe that
    \[
    \E_{(a, w, b)\sim\mu^\sharp}\big[|a|\big(\|w\|+ |b|\big)^\alpha\big]
    =  \E_{(a, w, b,t)}\big[|a|\big(|w| + |tw|+ |b|\big)^\alpha\big] \leq \E_{(a, w, b)\sim\mu}\big[|a|\big(|w| + |b|\big)^\alpha\big]\int_\R \frac{(1+|t|)^\alpha}{\pi(1+t^2)}\dt.
    \]
    We conclude that $u_\alpha\in \B_\alpha$. Since functions in $\B_\alpha$ grow sublinearly at infinity, we see that $u_\alpha$ is indeed the only solution to \eqref{eq harmonic half-space alpha} in $\B_\alpha$ by Liouville's Theorem.
\end{proof}

\begin{remark}
While we do not develop a functions space theory for it, we note Lemma \ref{lemma other representation} also applies to tanh activation, assuming that the magnitude of $|a|(|w|+1)^\alpha$ is controlled for some $\alpha\in[0,1)$. This norm bound suffices to control the $C^{0,\alpha}$-norm of a superposition of single neuron activations (or a `tanh-Barron function'), but not to control e.g.\ the Lipschitz-constant.

To see this, first note that $|a\,\tanh(w\cdot x + b)| \leq |a|$, i.e.\ the bound on $|a|(|w|+1)^\alpha$ yields an $L^\infty$-bound. Here the magnitude of the bias term is unconstrained, similar to \cite[Appendix A]{wojtowytsch2021kolmogorov}.

Additionally, note that $\sigma(x) := \tanh(x)$ satisfies
\[
-1 \leq\tanh\leq 1, \qquad \tanh(0) = 0, \qquad 0\leq \tanh' \leq 1.
\]
Hence, if $x,z\in\R$ we have
\[
\left| a\,\tanh(wx+b)- a\,\tanh(wz+b)\right| \leq |a|\,\big|(wx+b) - w(z+b)\big|\leq |a|\,|w(x-z)|.
\]
If $|w(x-z)|\leq 1$, this yields the bound
\[
\left| a\,\tanh(wx+b)- a\,\tanh(wz+b)\right| \leq |a|\,|w|^\alpha \,|x-z|^\alpha 
\]
while if $|w(x-z)| \geq 1$, then 
\[
\left| a\,\tanh(wx+b)- a\,\tanh(wz+b)\right| \leq 2|a| \leq 2\,|a|\,|w|^\alpha|x-z|^\alpha.
\]
Combining the two cases, we find a bound on the $C^{0,\alpha}$-semi-norm since the bound applies to any pair $x, z$. The bound can be extended to (continuous) superpositions of single neuron activations $a\,\tanh(wx+b)$ as in Section \ref{section barron background}.
\end{remark}

Finally, we note that the continuity of $g$ is only used to prove that the boundary values are attained uniformly (and this can be localized as in the proof of \cite[Theorem 2.2.14]{evans}. With lower regularity, the same obviously holds true for the classical Heaviside activation function $\sigma_0$ where the limit is attained at points $x^*\neq 0$. In this case, the integral can easily be computed directly:
\[
u(x,y) = \frac1\pi\int_\R\frac1{1+t^2} 1_{\{x/y+t>0\}} \dt = \frac1\pi \int_{-x/y}^\infty \frac1{1+t^2}\dt = \frac12 - \frac1\pi \arctan(-x/y) = \frac12 + \frac1\pi \arctan(x/y).
\]
In this way, $\alpha=0$ would behave more like $\alpha\in(0,1)$ than $\alpha \in \mathbb N\setminus\{0\}$.

\subsection{Integer powers of ReLU activation}
\label{section integer alpha}

In this section, we consider the case $\alpha\in \mathbb N\setminus\{0\}$. Following standard notation, we denote $k:= \alpha$.

\begin{lemma}\label{lemma uk is not barron}
If $k = \alpha\in\N$, then $u_\alpha$ is not an element of the $\ReLU^\alpha$-Barron space.
\end{lemma}

\begin{proof}
    {\bf Proof strategy:} If $u:\R^2_+\to\R$ is a Barron function, then also the one-dimensional `slice' $\tilde u: \R\to\R$ given by $\tilde u(t) = u(x_0+tv)$ is a Barron function for any $x_0, v\in\R^2$ by Lemma \ref{lemma dimension reduction 1}. We will demonstrate that $u_k$ is not a Barron function by showing that its restriction to a suitable slice is not Barron in one dimension. This is a much easier task, since we can use the characterization of Lemma \ref{lemma one dimension derivatives} and its Corollary \ref{corollary non-barron example}.

    Recall that
    \begin{align*}
        u_k(x,y) &= \left(\frac1\pi \,\arctan(x/y) + 1/2\right)\Re\big((x+iy)^k\big) - \frac1\pi\,\log\big(\sqrt{x^2+y^2}\big) \,\Im\big((x+iy)^k\big)\\
        &=: u^r_k (x,y) - u^i_k(x,y).
    \end{align*}
    We will show that $u_k^r$ is a ReLU$^k$-Barron function, while $u_k^i$ is not, even when considering just the restriction to $L$. This implies that $u_k = u^r_k - u^i_k$ is not in $\B_k$.

    {\bf Step 1. Non-Barron component.} To find a suitable slice, pick any $\theta$ such that $k\theta\notin \pi \Z$. Consider the line segment: 
    \[
    L = \{(x,y): y = x\tan{(\theta)} \text{ and } 0\leq x\leq 1\}.
    \]
    
     On the line segment $L$ we have
    \begin{align*}
    \tilde u^i_k(x) &= \frac{1}{\pi}\log{\left(\sqrt{x^{2}(1+\tan^{2}(\theta))}\right)}\left(\sqrt{x^{2}(1+\tan^{2}(\theta))}\right)^{k}\sin{(k\theta)}\\
    &= \frac{\sec^{k}{(\theta)}\sin{(k\theta)}}{\pi}\log{(x\sec{(\theta)})}x^{k}\\
    &= \frac{\sec^{k}{(\theta)}\sin{(k\theta)}}{\pi}\big(\log x + \log (\sec\theta)\big)x^{k}
    \end{align*}
    Since $k\theta\notin\pi \Z$, we find that the constant
    $c_{k\theta} = \sec^{k}(\theta)\sin{(k\theta)}/\pi$ is not zero. Additionally, we find that $x^k = \sigma_k(x)$ is in the $\sigma_k$-Barron space, while the second term $x\mapsto x^k\log x$ is not due to Corollary \ref{corollary non-barron example}.
    Thus, $u^i_k$ is not in $\B_k$ when restricted to $L$ and hence not in $\B_k(B_r(0))$ for any $r>0$.

    {\bf Step 2. Barron component.} On the other hand, let us show that $u^r_k$ is a Barron function on the whole space. We observe that $u^r_k(\lambda x, \lambda y) = \lambda^k u^r_k(x, y)$ for $\lambda >0$, so by Lemma \ref{lemma dimension reduction 2} it suffices to prove that
    \[
    \tilde u^r_k(\xi) := \left(\frac12 +\frac{\arctan(\xi)}\pi\right)\Re((\xi+i)^k)
    \]
    is in $\B_\alpha(\R)$. Again, we can use the criterion of Lemma \ref{lemma one dimension derivatives}. Naturally, it suffices to consider the function $\arctan(\xi) \,\Re((\xi+i)^k)$ since the polynomial component has $k+1$-th derivative zero and is automatically $\B_k$-Barron. The $k+1$-th derivative can be found in a direct, but lengthy computation. The details are given in Appendix \ref{appendix tedious calculation}.
\end{proof}

\begin{corollary}
    Let $\Omega\subseteq \R^2$ be open and $0\in\Omega$. Then there exists no solution $u\in \B_k(\Omega \cap \R^2_+)$ of the partial differential equation
    \begin{equation}\label{eq pde corr 310}
    \begin{pde}
        -\Delta u&=0 &\text{in }\Omega \cap \R^2_+\\
        u &= \sigma_k(x_1) &\text{on }\Omega \cap \partial \R^2_+.
    \end{pde}
    \end{equation}
\end{corollary}

\begin{proof}
We have seen that $u_k\notin \B_k$. If $u$ is any other solution of \eqref{eq pde corr 310}, then $-\Delta (u-u_k) = 0$ in $B_r(0)\cap \R^2_+$ and $u-u_k=0$ on $B_r(0)\cap \partial\R^2_+$ for $r$ so small that $B_r(0)\subseteq \Omega$.

Arguing by odd reflection as in the proof of Theorem \ref{theorem liouville}, we find that $u-u_k\in C^\infty(B_r(0))$. In particular, if we choose a cut-off function $\eta\in C_c^\infty(\R^2)$ such that $\eta \equiv 1$ on $B_{r/2}(0)$ and $\eta \equiv 0$ outside $B_r(0)$, then $\eta u\in C_c^\infty(\R^2)$. By Theorem \ref{theorem fractional sobolev embedding} can conclude that there exists $W\in\B_k$ such that $W\equiv u\eta \equiv u$ in $B_{r/2}$. 

If $u\in\B_k(\Omega\cap \R^2_+)$, then $u = u_k + (u-u_k) = u_k + W$ on $B_{r/2}$, meaning that $u_k = u-W \in \B_k(B_{r/2})$, which contradicts Lemma \ref{lemma uk is not barron}.
\end{proof}

While $u_k$ is {\em not} a Barron function, we now show that it can be approximated well by Barron functions of low Barron norm in various norms. In the following, we denote
\[
u_{\eps, k}(x,y) = \left(\frac12 + \frac{\arctan(x/y)}\pi\right)\Re\big((x+iy)^k\big) -\frac{\log(x^2+y^2+\eps^2)}{2\pi}\,\Im\big((x+iy)^k\big) =: u^r_k(x,y) - u^i_{\eps,k}(x,y).
\]
In particular, we have $u_k = u_{0,k}$.

\begin{theorem}
Let $R>0$ and $k\in \mathbb N$, $k\geq 2$. There exists $\eps_0>0$ depending on $R$ and $k$ such that
\begin{enumerate}
    \item The bound 
    \[
    \|u_{\eps, k}\|_{\B_k(B_R(0))}\leq C|\log\eps|
    \]
    holds for all $k\in\N$ and $\eps\in(0,\eps_0)$ and a constant $C$ which depends on $k$ and $R$, but not $\eps$.
    
    \item For $k\geq 2$, there exist constants $c,C>0$ such that for all $p\in[1,\infty]$ and all sufficiently small $\eps\in(0,\eps_0)$, we have
    \[
    c\,R^{k-2}\eps^2\leq \|u_{\eps,k}- u_k\|_{L^p(B_R(0)} \leq C\, R^{k-2}\eps^2.
    \]

    \item There exists a constant $C>0$ such that for all $p\in[1,\infty]$ and all $\eps\in(0,\eps_0)$, we have
    \[
    \|\nabla (u_{\eps,k}- u_k)\|_{L^p(B_R(0)} \leq C\eps^{2/p}.
    \]

    \item There exists a constant $C>0$ depending on $k, p, R$ such that for all $p\in[1,\infty]$ and all $\eps\in(0,\eps_0)$, we have 
    \[
    \|D^2(u_{\eps,k}-u_k)\|_{L^p(B_R(0)} \leq C\eps^{2/p}
    \]
    if $k>2$ or $p>1$ and $ \|D^2(u_{\eps,k}-u_k)\|_{L^p(B_R(0)} \leq C\eps^2|\log\eps|$ if $k=2$ and $p=1$.
\end{enumerate}
\end{theorem}

The case $k=1$ is considered separately in \cite{pde_relu}. The rate of convergence in $L^\infty$ is slower (namely $\eps$ rather than $\eps^2$), and the harmonic function $u_k$ fails to be $H^2$-regular, so additional limits are placed on the approximation of the Hessian.

\begin{proof}

\textbf{\bf Step 1a. Approximation in the $L^{\infty}$ norm.} 
% We begin by establishing that $u_k$ can be approximated well by Barron functions in the $L^\infty$-norm.
Denote the difference $v_{\varepsilon,k} := u_{\varepsilon,k} - u_{0,k}$. For simplicity, we use polar coordinates. Since $|\log(r^2) - \log(r^2+\eps^2)| = \log(1+\eps^2/r^2)>0$, we have
\[
||v_{\varepsilon,k}||_{L^{\infty}(B^{+}_{R}(0))} 
    = \max_{0\leq r\leq R}\max_{0\leq\phi\leq\pi}\left(\log{\left(1 + \frac{\varepsilon^{2}}{r^{2}}\right)}r^{k}|\sin{(k\phi)}|\right)
    = \max_{0\leq r\leq R}\left(r^{k}\log{\left(1 + \frac{\varepsilon^{2}}{r^{2}}\right)}\right)
\]
The function on the right is continuous on $[0,R]$ and differentiable in $(0,R)$, so the maximum is attained. A priori, it can be attained either at an endpoint $r=0$ or $r=R$ or at an intermediate point at which $\frac{d}{dr} r^k \log(1+\eps^2/r^2)=0$. We consider all three cases separately:

\begin{enumerate}
\item If $r=0$, then $r^k\log(1+\eps^2/r^2) = 0$ by continuous extension. The maximum is not attained here.

\item If $r =R$, then
\[
r^k\log(1+\eps^2/r^2) = R^k \log\left(1+\eps^2/R^2\right) \sim R^{k-2}\eps^2
\]
if $\eps$ is sufficiently small compared to $1/R$ by Taylor expansion $\log(1+ z) \approx z$.

\item If the extremum is attained at an interior point $0<r<R$, then
\begin{align*}
0 &= \frac{d}{d r}r^k\log\left(1+\frac{\eps^2}{r^2}\right)\\ 
    &= kr^{k-1}\log{\left(1 + \frac{\varepsilon^{2}}{r^{2}}\right)} + r^{k}\frac{1}{1+ \frac{\varepsilon^{2}}{r^{2}}}\left(\frac{-2\varepsilon^{2}}{r^{3}}\right) 
    \\&= r^{k-1}\left(k\log{\left(1 + \frac{\varepsilon^{2}}{r^{2}}\right)} - \frac{2\varepsilon^{2}}{r^{2}+ \varepsilon^{2}}\right).
\end{align*}
Thus, for such a point we have
\[
\frac{d}{d r}\left(r^k\log\left(1+\frac{\eps^2}{r^2}\right)\right) = 0 \qquad\Rightarrow \quad  \log{\left(1 + \frac{\varepsilon^{2}}{r^{2}}\right)} = \frac{2}{k}\,\frac{\varepsilon^{2}}{r^{2}+ \varepsilon^{2}}.
\]
In this case, assuming that $k\geq 2$, we have
\[
r^k\log\left(1+ \frac{\eps^2}{r^2}\right) = r^k\cdot\frac2k \frac{\eps^2}{r^2+\eps^2}\leq \frac2k\,\eps^2r^{k-2} \leq \frac2k \,\eps^2R^{k-2}
\]
since $r\mapsto r^{k-2}$ is monotone increasing in $r$. This bound coincides with the one close to the boundary of the disk (up to constant). 
\end{enumerate}

{\bf Step 1b. Lower bound in $L^1$.} Let us argue that the rate of approximation cannot be improved by relaxing the order of integrability. Namely, since
\[
|v_{\eps,k}(r,\phi)| = \left|r^k \log\left(1+ \frac{\eps^2}{R^2}\right)\sin(k\phi)\right| \geq \left(\frac R2\right)^k \,\frac12\,\left(\frac{\eps^2}{R^2}\right) |\sin(k\phi)| = 2^{-(k+1)}R^{k-2}\eps^2\,|\sin(k\phi)|,
\]
for $r\in(R/2, R)$ and sufficiently small $\eps>0$, we see that $\|v_{\eps,k}\|_{L^1(B_R^+)} \geq cR^{k-2}\eps^2$ for a universal constant $c>0$.

\textbf{Step 2. Closeness of Gradients.} Again, we work in polar coordinates and note that
\[
\int_{B_R^+}\|\nabla v\|^p\dx = \int_0^\pi \int_0^R \left((\partial_r v)^2 + \left(\frac1r \,\partial_\phi v\right)^2\right)^\frac p2\, r\,\dr\d\phi.
\]
Denoting the unit vectors $e_r = (\cos \phi, \sin \phi)$ and $e_\phi = (-\sin\phi, \cos\phi)$, we have
\begin{align*}
\nabla v_{\varepsilon,k} &= \nabla u_{\epsilon,k} - \nabla u_{0,k}\\
    % &= \nabla_{x,y}\Im{((x+iy)^{k})}\log{\left(1 + \frac{\varepsilon^{2}}{x^{2}+y^{2}}\right)}\\
    &=\nabla\left(r^{k}\sin{(k\phi)}\log{\left(1 + \frac{\varepsilon^{2}}{r^{2}}\right)}\right)\\
    &= \left\{\sin{(k\phi)}\,\partial_{r} \left(r^{k}\log{\left(1 + \frac{\varepsilon^{2}}{r^{2}}\right)} \right)\right\} \,e_r + \left\{\frac{r^{k}}{r}\log\left(1 + \frac{\varepsilon^{2}}{r^{2}}\right)\partial_{\phi}\sin(k\phi)\right\}\, e_\phi\\
    &= \left\{r^{k-1}\sin{(k\phi)}\left(k\log{\left(1 + \frac{\varepsilon^{2}}{r^{2}}\right)} - \frac{2\varepsilon^{2}}{r^{2}+\varepsilon^{2}}\right)\right\}e_r + \left\{kr^{k-1}\log{\left(1 + \frac{\varepsilon^{2}}{r^{2}}\right)}\cos{(k\phi)}\right\}e_\phi.
\end{align*}
We can estimate the separate terms occurring in the gradient separately. As before, we bound $\sin, \cos$ from above by $1$.
Then we have
\begin{align*}
\int_{B_R^+} \left|kr^{k-1}\log{\left(1 + \frac{\varepsilon^{2}}{r^{2}}\right)}\cos{(k\phi)}\right|^{p}\dx 
    &\leq k^{p}\pi\int_{0}^{R}r^{p(k-1)+1}\log^{p}{\left(1 + \frac{\varepsilon^{2}}{r^{2}}\right)}\dr\\
    &= k^{p}\pi\varepsilon^{p(k-1)+2}\int_{0}^{R}\log^{p}{\left(1 + \left(\frac{\varepsilon}{r}\right)^2\right)}\left(\frac r\eps\right)^{p(k-1)+1}\,\frac1\eps\dr\\
    &= k^{p}\pi\varepsilon^{p(k-1)+2}\int_{\varepsilon/R}^{\infty}\frac{\log^{p}{(1+z^{2})}}{z^{p(k-1)+3}}\dz
\end{align*}
with the substitution $z= \frac{\varepsilon}{r}$.
The integral can be estimated further by
\begin{align*}
\int_{\eps/R}^\infty \frac{\log^p(1+z^2)}{z^{(k-1)p+3}}\dz
    &\leq \int_{\varepsilon/R}^{5}\frac{\log^{p}{(1+z^{2})}}{z^{p(k-1)+3}}dz + \int_{5}^{\infty}\frac{\log^{p}{(1+z^{2})}}{z^{p(k-1)+3}}\dz\\
    &\leq \int_{\varepsilon/R}^5 z^{2p-p(k-1)+3}\dz + 3^p\int_{5}^{\infty}\left(\frac{\log(z^{k-1})}{z^{k-1}}\right)^p\frac1{z^3}\dz
\end{align*}
since $\log(1+z^2) \leq \log(2z^2) \leq \log(z^3) = 3\log(z) \leq 3\log(z^{k-1})$ for $z\geq 2$ and $k\geq 2$. For the lower integral, we used the fact that $\log(1) = 0$ and $\log'(\xi)\leq 1$ for $\xi\geq 1$, so $\log(1+z^2)\leq z^2$. The function $(\log \xi)/\xi$ is monotone decreasing on $[e, \infty)$ and $\log 5/5 \leq \frac13$, so 
\[
\int_{\eps/R}^\infty \frac{\log^p(1+z^2)}{z^{(k-1)p+3}}\dz \leq 
    \int_{\eps/R}^2 z^{p(3-k)+3}\dz + \left(\frac{3\log 5}5\right)^p \int_5^\infty \frac1{z^3}\dz \lesssim \left(\frac{\eps}R\right)^{p(3-k)+4} + \int_5^\infty \frac1{z^3}\dz
\]
if $p(3-k)+3 < 0$ so that the {\em lower} bound of the bottom integral dominates. Otherwise, the first integral is bounded by $5^{p(3-k)+4}$, which is constant in $\eps$. Overall:
\[
\int_{0}^R \left|k\,r^{k-1}\,\log\left(1+\frac{\eps^2}{r^2}\right)\right|^p\,r\dr \leq C \eps^{p(k-1)+2}(\eps^{4 + p(3-p)}+ 1) \sim \eps^{6+2p} + \eps^{2+p(k-1)}.
\]
Let us estimate the second term:
\begin{align*}
\int_{\Omega} \left|\frac{2\varepsilon^{2}}{r^{2}+\varepsilon^{2}}r^{k-1}\sin{(k\phi)}\right|^{p}rd\phi \dr       &\leq 2^{p}\pi\varepsilon^{2p}\int_{0}^{R}\frac{r^{p(k-1)+1}}{\left(r^{2}+\varepsilon^{2}\right)^{p}}\dr\\
    &= 2^{p}\pi\varepsilon^{2p}\left(\int_{0}^{\varepsilon}\frac{r^{p(k-1)+1}}{(r^{2}+\varepsilon^{2})^{p}}dr + \int^{R}_{\varepsilon}\frac{r^{p(k-1)+1}}{(r^{2}+\varepsilon^{2})^{p}}dr \right)\\
    &\leq 2^{p}\pi\varepsilon^{2p}\left(\frac{1}{\varepsilon^{2p}}\int_{0}^{\varepsilon}r^{p(k-1)+1}dr +\int_{\varepsilon}^{R}r^{p(k-3)+1}dr\right)\\
    &\lesssim \eps^{p(k-1)+2} + \eps^{2p + p(k-3)+2}+\eps^{2p}
\end{align*}
The estimate holds both for $k\geq 3$ and $k\leq 3$ in the given formulation. Combining the estimates, we obtain the bound on gradient closeness.

{\bf Step 3. Barron norm estimates.} Since $u_k^r$ is a Barron function, it suffices to estimate the Barron norm of $u_{\eps, k}^i$, i.e.\ of
\[
2\pi \,u_{\eps,k}^i(x,y) = \log(x^2+y^2+\eps^2)\,P_k(x,y)\qquad\text{where } P_k(x,y) = \Im((x+iy)^k)
\]
is a homogeneous polynomial of degree $k$. We note that
\[
\frac{\partial^l}{\partial x^l} \frac{\partial^m}{\partial y^m}2\pi \,u_{\eps,k}^i(x,y) = \sum_{i=0}^l \sum_{j=0}^m \binom{l}i \binom{m}j\,\partial_x^i \,\partial_y^j\log(x^2+y^2+\eps^2)\cdot \partial_x^{l-i}\,\partial_y^{m-j}P_k(x,y).
\]
We note that $p_{k, l-i, m-j}:= \partial_x^{l-i}\,\partial_y^{m-j}P_k$ is a homogeneous polynomial of degree $k- (l-i)-(m-j)$ in $(x,y)$ and claim that 
\[
\partial_x^i \,\partial_y^j\log(x^2+y^2+\eps^2) = \sum_{s=1}^{i+j} \frac{q_{k,i,j, s}(x,y)}{\big(x^2+y^2+\eps^2)^s}
\]
for some homogeneous polynomial $q_{k,i,j,s}$ of degree $2s-i-j$ if $i+j>0$ (with coefficients which do not depend on $\eps$). This is easy to see if $i+j=1$ with
\[
\partial_x\,\log(x^2+y^2+\eps^2) = \frac{2x}{x^2+y^2+\eps^2}, \qquad
\partial_y\,\log(x^2+y^2+\eps^2) = \frac{2y}{x^2+y^2+\eps^2}, \qquad 2s - (i+j) = 2-1 =1
\]
and follows by induction on $i+j$ in general. For instance:
\begin{align*}
\partial_x^{i+1} \,\partial_y^j\log(x^2+y^2+\eps^2) &= \partial_x \sum_{s=1}^{i+j} \frac{q_{k,i,j, s}(x,y)}{\big(x^2+y^2+\eps^2)^s}\\
    &= \sum_{s=1}^{i+j} \left(\frac{\partial_x q_{k,i,j,s}}{(x^2+y^2+\eps^2)^s} - \frac{2s\,x\,q_{k,i,j,s}}{(x^2+y^2+\eps^2)^{s+1}}\right)\\
    &= \sum_{s=1}^{i+j+1} \frac{q_{k,i+1,j, s}(x,y)}{\big(x^2+y^2+\eps^2)^s}
\end{align*}
with
\[
q_{k,i+1,j, s} = \partial_x q_{k,i,j,s} - 2(s-1)x\,q_{k,i,j, s-1}.
\]
The claim follows by noting that the degree of $q_{k, i, j,s}$ is indeed a homogeneous polynomial of degree one lower than that of $q_{k,i,j,s}$ -- i.e.\ $2s - (i+1) -j$ -- and one higher than that of $q_{k,i,j,s}$ -- i.e.\ $2(s-1) -i - j+1 = 2s - (i+1) -j$.

Consider the case $l+m \geq k+1$. Then, if $i = j = 0$ we have $p_{k, l-i, m-j} = \partial_x^{l-i}\partial_y^{m-j} P_k = 0$ since $P_k$ has degree $k$. We therefore have
\[
\partial_x^l \,\partial_y^m u_{\eps,k^i}(x,y) = \sum_{s=1}^{l+m} \frac{\tilde P_{k,l,m,s}(x,y)}{(x^2+y^2+\eps^2)^s}
\]
where 
\[
\tilde P_{k,l,m,s} = \sum_{i=0}^l\sum_{j=0}^m \binom li \binom mj\,p_{k, l-i, m-j}\,q_{k,i,j,s}
\]
is a homogeneous polynomial of degree $k+2s - (l+m)$ which does not depend on $\eps>0$. In particular, if $l+m \geq k+1$, we have
\begin{align*}
\int_{B_R(0)}\left|\partial_x^l \,\partial_y^m u_{\eps,k^i}\right|^2\,\dx \dy
    &\leq C \int_0^R \left|\sum_{s=1}^{k+l}\frac{r^{k+2s - l -m}} {(r^2 + \eps^2)^{s}}\right|^2r \dr\\
    &\leq C \max_{s\geq 1} \int_0^R \frac{r^{k+3-l-m}}{r^2+\eps^2}\left(\frac{r^2}{r^2+\eps^2}\right)^{s-1}\dr\\
    &= C\int_0^R \frac{r^{k+3-l-m}}{r^2+\eps^2}\dr\\
    &\leq C\left(\eps^{-2}\int_0^\eps r^{k+3-l-m}\dr + \int_\eps^r r^{k+1-l-m}\dr\right)\\
    &\leq \begin{cases} C \left(\eps^{k+2-l-m} + R^{k+2 -l-m} \right) &l+m \geq k+3\\
    C \left(1 + |\log\eps| + |\log R| \right) & k+m = k+2\end{cases}.
\end{align*}
For fixed $R>0$, we conclude that
\[
\|u_{\eps,k}^i\|_{H^{k+2}(B_R}^2 \leq C \big(1+|\log\eps|\big), \qquad \|u_{\eps,k}^i\|_{H^{k+3}(B_R}^2\leq \frac C\eps.
\]
We can use Theorem \ref{theorem fractional sobolev embedding} and interpolation between $H^{k+2}$ and $H^{k+3}$ to conclude that $\|u_{\eps,k}^i\|_{\B_k(B_r(0))} \leq C(k,r)\,|\log\eps|$ exactly as in \cite{pde_relu}.

{\bf Step 4. Closeness of the second derivatives.} The same argument can be used to consider the case $l+m = 2 \leq k$. Then
\begin{align*}
\int_{B_R(0)} & |D^2(u_{\eps,k}-u_k)|^p\dx \leq C^p \int_0^R \bigg\{\big( \log (r^2+\eps^2) - \log(r^2)\big)^pr^{(k-2)p} + \left|\frac r{r^2+\eps^2} - \frac1{r}\right|^pr^{(k-1)p}\\
    &\hspace{5.7cm}+ \left|\frac{1}{r^2+\eps^2} - \frac1{r^2}\right|^pr^{kp}+ \left|\frac{r^2}{(r^2+\eps^2)^2} - \frac1{r^2}\right|^pr^{kp} \bigg\}r\dr\\
    &\leq C^p\int_0^R\log^p\left(1+\frac{\eps^2}{r^2}\right)r^{(k-2)p+1} + \frac{\eps^{2p}r^{(k-2)p+1}}{(\eps^2+r^2)^p} + \frac{\eps^{2p}r^{kp+1}}{(\eps^2+r^2)^{2p}} + \frac{\eps^{4p}r}{(\eps^2+r^2)^{2p}} \dr\\
    &\leq  C^p\int_0^R\log^p\left(1+\frac{\eps^2}{r^2}\right)r^{(k-2)p+1} + \frac{\eps^{2p}r^{(k-2)p+1}}{(\eps^2+r^2)^p} + \frac{\eps^{4p}r}{(\eps^2+r^2)^{2p}} \dr
\end{align*}
for a constant $C>0$ that does not depend on $p$. We estimate the integrals separately:
\begin{align*}
\int_0^R \frac{\eps^{4p}r}{(\eps^2+r^2)^{4p}} \dr
    &\leq \int_0^\eps \frac{\eps^{4p}r}{\eps^{4p}} \dr + \int_\eps^R \eps^{4p} r^{1-4p}\dr\\
    &\leq \frac{\eps^2}2 + \frac{\eps^{2-4p}\eps^{4p}}{4p-2} \sim \eps^2\\
\int_0^R \frac{\eps^{2p}r^{(k-2)p+1}}{(\eps^2+r^2)^p} \dr
    &\leq \int_0^{\eps} r^{(k-2)p+1} \dr + \int_{\eps}^R\eps^{2p}r^{(k-4)p+1}\dr\\
    &\sim \eps^{(k-2)p+2} + \eps^{2p + (k-4) p + 2}\\
    &= \eps^{2 + (k-2)p}
\end{align*}
if $(k-4)p+1 <0$ -- otherwise, the second term is merely a constant.  Finally:
\begin{align*}
\int_0^R \log^p\left(1+\frac{\eps^2}{r^2}\right)&r^{(k-2)p+1} \dr
    = \eps^{(k-2)p+2}\int_0^R\log^p\left(1+\bigg(\frac{\eps}{r}\bigg)^2\right)\left(\frac r\eps\right)^{(k-2)p+3} \frac{\eps}{r^2}\dr\\
    &= \eps^{(k-2)p+2}\int_{\eps/R}^\infty \log^p(1+z^2) z^{-(3+(k-2)p)}\dz\\
    &\leq \eps^{(k-2)p+2} \left(\int_{\eps/R}^1z^{2p}z^{-(k-2)p-3}\dz + \int_1^\infty \log^p(1+z^2)z^{-3 - (k-2)p}\dz\right)\\
    &\leq \eps^{(k-2)p+2}\int_{\eps/R}^1 z^{(4-k)p-3}\dz + \eps^{(k-2)p+2}\int_1^\infty \frac{\log^p(1+z^2)}{z^3}\dz.
\end{align*}
In \cite{pde_relu}, the bound
\[
\int_1^\infty \frac{\log^p(1+z^2)}{z^3}\dz \leq \Gamma(p+1)
\]
is established, and since $(4-k)p-3 \geq 2p-3 >1$ unless $k=2$ and $p=1$, we observe that
\[
\eps^{(k-2)p+2}\int_{\eps/R}^1 z^{(4-k)p-3}\dz \sim \frac{\eps^{(k-2)p+2 + (4-k)p -2}}{R^{(4-k)p-2}} = \frac{\eps^{2p}}{R^{(4-k)p-2}}.
\]
If $k=2$ and $p=1$, the rate is $\eps^2|\log\eps|$ instead since $(4-k)p-3 = 2-3 = -1$. Thus overall
\begin{align*}
\int_{B_R(0)}  |D^2(u_{\eps,k}-u_k)|^p\dx &\leq C\left(\eps^2 + \eps^{(k-2)p+2} + \eps^{2p}\right) \leq C \eps^2
\end{align*}
for since $(k-2)p\geq 0$ and $2p\geq 1$. Again, we make an exception for $k=2$ and $p=1$.
\end{proof}

In particular, the quality of approximation in $H^1$ and $H^2$ is comparable for $k\geq 2$. From this point of view, we have no reason to prefer either the PINN method (strong formulation) or the Deep Ritz method (variational formulation) to solving the Poisson equation by neural networks.

\section{Open Questions, Limitations, Future Work}\label{section conclusion}

Several important questions remain open, chief among them: If $\alpha\in (1, \infty)\setminus \mathbb N$, is there a solution to the problem
\[
\begin{pde}
    -\Delta u &= u &\text{in }\R^2_+\\
    u &= \ReLU^\alpha(x_1) & \text{on }\partial \R^2_+
\end{pde}
\]
in the $\ReLU^\alpha$-Barron space? Or, if it is not, is the failure merely logarithmic as it is in the integer case? As illustrated above, this question is of interest for a wide class of neural PDE solvers, among them the Deep Ritz method, PINNs and DeepONets.

On the technical side, we note that we only established upper bounds on the order of approximation, but no lower bounds. Establishing matching upper and lower bounds is more subtle here than in many other applications since our bounds are of the form
\[
\min_{\|w\|_{\B_k}\leq s} \|u_k - w\|_{W^{m,p}} \leq C\,\exp\left(\gamma s\right)
\]
for some constants $C = C_{k,m,p}$ and $\gamma = \gamma_{k,m,p}$ depending on the power $k$ of the activation function and the norm of the $W^{m,p}$-space in which we are approximating $u_k$. The order of approximation is always exponential in $s$, but the speed may yet differ. 

For ease of presentations, we focussed on PDEs on the half space. Another important question is whether the same results apply to PDEs in compact domains. A positive answer is given in \cite{pde_relu} for partial differential equations on rectangular domains in two dimension, but not for more general domains or powers of ReLU activation.

Finally, we note that neural networks with more than one hidden layer are more popular in applications than the simple perceptron architecture studied in this article. Recently, several function space theories were proposed for multi-layer perceptron (MLP) architectures \cite{wojtowytsch2020banach, chen2024neural}. Whether and how the arguments of this analysis could be applied in this highly relevant setting remains open.

\section*{Acknowledgements}
SW gratefully acknowledges funding by the National Science Foundation through Award \# 2424801 which facilitated this research.

\appendix

\section{Proof of Lemma \ref{lemma one dimension derivatives}: One-dimensional Barron functions and derivatives}

\begin{proof}[Proof of Lemma \ref{lemma one dimension derivatives}]
    {\bf First direction.} Assume for the moment that $f^{(k+1)}\in L^1(\R)$ is a function such that
    \[
    \int_\R \big|f^{(k+1)}(x)\big|\,\big(1+|x|^k\big)\dx < +\infty.
    \]
    By the Cauchy formula for repeated integration (or Taylor's theorem with remainder in integral form), we have the Riemann-Liouville integral identity 
    \[
    f(x) = \sum_{i=0}^k f^{(i)}(0)\,x^i + \frac1{k!}\int_0^x (x-t)^i \,f^{(k+1)}(t)\dt = \sum_{i=0}^k f^{(i)}(0)\,x^i + \frac1{k!}\int_0^\infty f^{(k+1)}(t)\,\sigma_k(x-t)\dt
    \]
    for $x>0$ and more generally
    \[
    f(x) = \sum_{i=0}^k f^{(i)}(0)\,x^i + \frac1{k!}\int_0^\infty f^{(k+1)}(t)\,\sigma_k(x-t)\dt + \frac1{k!}\int_{-\infty}^0 f^{(k+1)}(t)\,\sigma_k(-x+t)\dt.
    \]
    Evidently, the integral contribution can be written in the form
    \[
    \mathbb E_{(a,w,b)\sim\pi} \big[a\,\sigma(w^Tx+b)\big]
    \]
    where $\pi= \phi_{1\sharp}\sharp \mu_1 + \phi_{2\sharp}\mu_2$ for
    \[
    \mu_1 = \frac{|f^{(k+1)}(t)|\, 1_{\{t<0\}}}{2 \int_{-\infty}^0|f^{(k+1)}(s)\ds} \cdot \mathcal L^1, \qquad \mu_1 = \frac{|f^{(k+1)}(t)|\, 1_{\{t>0\}}}{2 \int_{0}^\infty|f^{(k+1)}(s)\ds} \cdot \mathcal L^1
    \]
    and the maps $\phi_1, \phi_2:\R\to \R^3$ mapping $t$ to $(a,w,b)$ are given by
    \begin{align*}
    \phi_1(t) &= \left( \frac2{k!}\, sign(f^{(k+1})(t))\,1_{\{t<0\}} \int_{-\infty}^0|f^{(k+1)}(s)\ds, \:-1, \:t\right)\\
    \phi_2(t) &= \left( \frac2{k!}\, sign(f^{(k+1})(t))\,1_{\{><0\}} \int_{0}^\infty|f^{(k+1)}(s)\ds, \:1,\: -t\right).
    \end{align*}
    Since the bound
    \[
    \mathbb E_{(a,w,b)\sim\pi} \big[|a|\,\big(|w|+ |b|)^k\big] = \int_\R |f^{(k+1)}(t)\big(1+|t|\big)^k\dt
    \]
    holds by assumption, the integral part is a $\sigma_k$-Barron function. The same is true for the finite sum since $x^k = \sigma_k(x) + (-1)^{k} \,\sigma_k(-x)$ and $x\mapsto x^j= \frac{d^{k-j}}{dx^{k-j}}x^k$, i.e.\ $x^j$ is in the closure of the subspace $span(\{(x-h)^k : h \in\R\}$ of the space of polynomials of degree at most $k$. Since the space is finite-dimensional, it is closed, and $x^j$ is an element of $\B_k$. For a more quantitative argument, see \cite[Proposition 2.1]{heeringa}.

    The same is true by approximation if $f^{(k+1)} \cdot \L^1$ is replaced by a more general Radon measure, i.e.\ if the $k+1$-th derivative $f^{(k+1)}$ is merely a Radon measure, not a function.

    {\bf Second direction.} Note that
    \[
    \frac{d^k}{dx^k}\sigma_k(wx+b) = k!\,w^k\,\sigma_0(wx+b)
    \]
    i.e.\ in the generalized sense
    \[
    \int_\R \left|\frac{d^{k+1}}{dx^{k+1}}\sigma_k(wx+b)\right|\dx = k!\,|w|^k
    \]
    since $\lim_{x\to\infty}\big|\sigma_0(wx+b) - \sigma_0(-wx+b)\big| = 1$. Additionally
    \[
    \int_\R \left|\frac{d^{k+1}}{dx^{k+1}}\sigma_k(wx+b)\right|\,|x|^k\dx = \int_\R \left|\frac{d}{dx}\sigma_0(wx+b)\right| |wx|^k\dx.
    \]
    For $w\neq 0$, the (distributional) derivative of $\sigma_0(wx+b)$ is a Dirac mass $\delta_{-b/w}$ since the `jump' in $\sigma_0$ is of height one independently of the argument. Thus:
    \[
    \int_\R \left|\frac{d^{k+1}}{dx^{k+1}}\sigma_k(wx+b)\right|\,|x|^k\dx = |b|^k
    \]
    unless $w=0$. In particular, we find that
    \[
    \int_\R \left|\frac{d^{k+1}}{dx^{k+1}}\sigma_k(wx+b)\right|\dx = |w|^k+|b|^k. 
    \]
    If $f(x) = \E_{(a,w,b)\sim\pi} \big[a\,\sigma_k(wx+b)\big]$, we heuristically sketch that
    \begin{align*}
    \int_\R \big|f^{(k+1)}(x)\big|\dx 
        &= \int_\R \left|\frac{d^{k+1}}{dx^{k+1}} \E_\pi \big[a\sigma_k(wx+b)\big] \right| \dx &&\leq \E_\pi\bigg[|a|\int_\R \left|\frac{d^{k+1}}{dx^{k+1}} \sigma_k(wx+b) \right| \dx\bigg]\\
        &= k!\,\E_\pi \big[|a|\big(|w|^k+|b|^k\big)] &&\sim k!\,\E_\pi \big[|a|\big(|w|+|b|\big)^k\big].
    \end{align*}
    More technically, we observe that the distributional derivative $f^{(k+1)}$ is the measure given by the push-forward
    \[
    \phi_\sharp \big(a\,|w|^k\big)\cdot \pi, \qquad \phi(a,w,b) = - \frac{b}{|w|}.
    \]
    A detailed version of the argument can be found in \cite[Example 4.1]{representationformulas}.
\end{proof}

\section{Proof of Lemma \ref{lemma uk is not barron}: Analysis of Barron component}\label{appendix tedious calculation}

Recall that the $l$-th derivative of $\arctan$ is
\[
\arctan^{(l)}(x) = (-1)^{l-1} \frac{(x+iy)^{-l} - (x-iy)^{-l}}{2i}\,(l-1)!
\]
Using this, we compute the $k+1$-th derivative of
\begin{align*}
f_k(x) &= \arctan(x/y)\,\frac{(x+iy)^k + (x-iy)^k}2
\end{align*}
as
\begin{align*}
\partial_x^{k+1}\,f_k(x) &= \sum_{l=0}^{k+1} \binom{k+1}l \,\left(\partial_x^l \arctan(x/y)\right) \left(\partial_x^{k+1-l}\frac{(x+iy)^k + (x-iy)^k}2\right)\\
	&= \sum_{l=1}^{k+1} \binom{k+1}l \,y^{-l} \arctan^{(l)}(x/y) \left(\prod_{j=l}^k l\right)\frac{(x+iy)^{l-1} + (x-iy)^{l-1}}2 \\
	&= \sum_{l=1}^{k+1} \binom{k+1}l\,(-1)^{l-1}\,(l-1)!\frac{(x/y+i)^{-l} - (x/y-i)^{-l}}{2i}\frac{k!}{(l-1)!} \frac{(x+iy)^{l-1} + (x-iy)^{l-1}}{2y^l}\\
	&= \frac{k!}{4i} \sum_{l=1}^{k+1} (-1)^{l-1} \binom{k+1}l \big\{(x+iy)^{l-1} + (x-iy)^{l-1}\big\}\big\{(x+iy)^{-l}- (x-iy)^{-l}\big\}.
\end{align*}
Note that
\begin{align*}
\big(z^{l-1} + \bar z^{l-1}\big)\big(z^{-l} - \bar z^{-l}\big) &= z^{-1} + \bar z^{l-1}z^{-l} - z^{l-1}\bar z^{-l} - \bar z^{-1}\\
	&= \frac1r \big(e^{-i\phi}  + e^{(1-2l)i\phi} - e^{(2l-1)i\phi} - e^{i\phi} \big)
\end{align*}
and that
\begin{align*}
\sum_{l=1}^{k+1} (-1)^{l-1} \binom{k+1}l &= -\left(\sum_{l=0}^{k+1} \binom{k+1}{l} (-1)^l 1^{k+1-l} - 1\right) = - \left((1-1)^{k+1} -1\right) = 1\\
\sum_{l=1}^{k+1} (-1)^{l-1} \binom{k+1}l e^{-2li\phi} &= -\left(\sum_{l=0}^{k+1} \binom{k+1}{l} (-e^{-2i\phi})^l 1^{k+1-l} - 1\right) = -\left(\big(1-e^{-2i\phi}\big)^{k+1}-1\right)
\end{align*}
so
\begin{align*}
\partial_x^{k+1}\,f_k(x) &= - \frac{k!}{4i\,}\left(e^{-i\phi} - e^{i\phi} \big\{\big(1-e^{-2i\phi}\big)^{k+1}-1\big\} + e^{-i\phi}\big\{\big(1-e^{2i\phi}\big)^{k+1} -1\big\} - e^{i\phi} \right)\\
	&= - \frac{k!}{4i\,r} \left( e^{i\phi}\big(1-e^{-2i\phi}\big)^{k+1} - e^{-i\phi}\big(1-e^{2i\phi}\big)^{k+1}\right)\\
	&= \frac{k!}{2r} \,\Im \big(e^{i\phi}(1-e^{-2i\phi})^{k+1}\big)
\end{align*}
where $r = \sqrt{x^2+y^2}$ and $\phi = \cot^{-1} (x/y)$. The quotient is well-defined since $y\neq 0$ in $\R^2_+$. Note that
\[
e^{i\phi}(1-e^{-2i\phi})^{k+1} = \big(1 - i\phi+O(\phi^2)\big) \left(-2i\phi + \frac {\phi^2}2 + O(\phi^3)\right)^{k+1} = O(\phi^{k+1}).
\]

\bibliographystyle{alpha}
\bibliography{./bibliography.bib}

\newcommand{\etalchar}[1]{$^{#1}$}
\begin{thebibliography}{CDCG{\etalchar{+}}22}

\bibitem[AKL{\etalchar{+}}24]{azizzadenesheli2024neural}
Kamyar Azizzadenesheli, Nikola Kovachki, Zongyi Li, Miguel Liu-Schiaffini, Jean
  Kossaifi, and Anima Anandkumar.
\newblock Neural operators for accelerating scientific simulations and design.
\newblock {\em Nature Reviews Physics}, pages 1--9, 2024.

\bibitem[Bar93]{barron1993}
Andrew~R Barron.
\newblock Universal approximation bounds for superpositions of a sigmoidal
  function.
\newblock {\em IEEE Transactions on Information theory}, 39(3):930--945, 1993.

\bibitem[BF09]{busam2009complex}
Rolf Busam and Eberhard Freitag.
\newblock {\em Complex analysis}.
\newblock Springer, 2009.

\bibitem[BF23]{boursier2023penalising}
Etienne Boursier and Nicolas Flammarion.
\newblock Penalising the biases in norm regularisation enforces sparsity.
\newblock {\em Advances in Neural Information Processing Systems},
  36:57795--57824, 2023.

\bibitem[CDCG{\etalchar{+}}22]{cuomo2022scientific}
Salvatore Cuomo, Vincenzo~Schiano Di~Cola, Fabio Giampaolo, Gianluigi Rozza,
  Maziar Raissi, and Francesco Piccialli.
\newblock Scientific machine learning through physics--informed neural
  networks: Where we are and what’s next.
\newblock {\em Journal of Scientific Computing}, 92(3):88, 2022.

\bibitem[Che24]{chen2024neural}
Zhengdao Chen.
\newblock Neural hilbert ladders: Multi-layer neural networks in function
  space.
\newblock {\em Journal of Machine Learning Research}, 25(109):1--65, 2024.

\bibitem[CLL21]{chen2021representation}
Ziang Chen, Jianfeng Lu, and Yulong Lu.
\newblock On the representation of solutions to elliptic pdes in barron spaces.
\newblock {\em Advances in neural information processing systems},
  34:6454--6465, 2021.

\bibitem[CMW{\etalchar{+}}21]{cai2021physics}
Shengze Cai, Zhiping Mao, Zhicheng Wang, Minglang Yin, and George~Em
  Karniadakis.
\newblock Physics-informed neural networks (pinns) for fluid mechanics: A
  review.
\newblock {\em Acta Mechanica Sinica}, 37(12):1727--1738, 2021.

\bibitem[CPV23]{caragea2023neural}
Andrei Caragea, Philipp Petersen, and Felix Voigtlaender.
\newblock Neural network approximation and estimation of classifiers with
  classification boundary in a barron class.
\newblock {\em The Annals of Applied Probability}, 33(4):3039--3079, 2023.

\bibitem[Cyb89]{cybenko1989approximation}
George Cybenko.
\newblock Approximation by superpositions of a sigmoidal function.
\newblock {\em Mathematics of control, signals and systems}, 2(4):303--314,
  1989.

\bibitem[DJL{\etalchar{+}}21]{duan2021convergence}
Chenguang Duan, Yuling Jiao, Yanming Lai, Xiliang Lu, and Zhijian Yang.
\newblock Convergence rate analysis for deep ritz method.
\newblock {\em arXiv preprint arXiv:2103.13330}, 2021.

\bibitem[DMZ22]{dondl2022uniform}
Patrick Dondl, Johannes M{\"u}ller, and Marius Zeinhofer.
\newblock Uniform convergence guarantees for the deep ritz method for nonlinear
  problems.
\newblock {\em Advances in Continuous and Discrete Models}, 2022(1):49, 2022.

\bibitem[Dob10]{dobrowolski2010angewandte}
Manfred Dobrowolski.
\newblock {\em Angewandte Funktionalanalysis: Funktionalanalysis,
  Sobolev-R{\"a}ume und elliptische Differentialgleichungen}.
\newblock Springer-Verlag, 2010.

\bibitem[Eva22]{evans}
Lawrence~C Evans.
\newblock {\em Partial differential equations}, volume~19.
\newblock American Mathematical Society, 2022.

\bibitem[EW20a]{wojtowytsch2020banach}
Weinan E and Stephan Wojtowytsch.
\newblock On the banach spaces associated with multi-layer relu networks:
  Function representation, approximation theory and gradient descent dynamics.
\newblock {\em arXiv preprint arXiv:2007.15623}, 2020.

\bibitem[EW20b]{wojtowytsch2020some}
Weinan E and Stephan Wojtowytsch.
\newblock Some observations on high-dimensional partial differential equations
  with barron data.
\newblock {\em arXiv preprint arXiv:2012.01484}, 2020.

\bibitem[EW21]{wojtowytsch2021kolmogorov}
Weinan E and Stephan Wojtowytsch.
\newblock Kolmogorov width decay and poor approximators in machine learning:
  Shallow neural networks, random feature models and neural tangent kernels.
\newblock {\em Research in the mathematical sciences}, 8(1):1--28, 2021.

\bibitem[EW22]{representationformulas}
Weinan E and Stephan Wojtowytsch.
\newblock Representation formulas and pointwise properties for barron
  functions.
\newblock {\em Calculus of Variations and Partial Differential Equations},
  61(2):46, 2022.

\bibitem[Hor91]{hornik1991approximation}
Kurt Hornik.
\newblock Approximation capabilities of multilayer feedforward networks.
\newblock {\em Neural networks}, 4(2):251--257, 1991.

\bibitem[Hor93]{hornik1993some}
Kurt Hornik.
\newblock Some new results on neural network approximation.
\newblock {\em Neural networks}, 6(8):1069--1072, 1993.

\bibitem[HS99]{hile1999gradient}
GN~Hile and Alexander Stanoyevitch.
\newblock Gradient bounds for harmonic functions {L}ipschitz on the boundary.
\newblock {\em Applicable Analysis}, 73(1-2):101--113, 1999.

\bibitem[HSSB24]{heeringa}
Tjeerd~Jan Heeringa, Len Spek, Felix~L Schwenninger, and Christoph Brune.
\newblock Embeddings between {B}arron spaces with higher-order activation
  functions.
\newblock {\em Applied and Computational Harmonic Analysis}, 73:101691, 2024.

\bibitem[JLL{\etalchar{+}}21]{jiao2021error}
Yuling Jiao, Yanming Lai, Yisu Lo, Yang Wang, and Yunfei Yang.
\newblock Error analysis of deep ritz methods for elliptic equations.
\newblock {\em arXiv preprint arXiv:2107.14478}, 2021.

\bibitem[KLL{\etalchar{+}}23]{kovachki2023neural}
Nikola Kovachki, Zongyi Li, Burigede Liu, Kamyar Azizzadenesheli, Kaushik
  Bhattacharya, Andrew Stuart, and Anima Anandkumar.
\newblock Neural operator: Learning maps between function spaces with
  applications to pdes.
\newblock {\em Journal of Machine Learning Research}, 24(89):1--97, 2023.

\bibitem[LJK19]{lu2019deeponet}
Lu~Lu, Pengzhan Jin, and George~Em Karniadakis.
\newblock Deeponet: Learning nonlinear operators for identifying differential
  equations based on the universal approximation theorem of operators.
\newblock {\em arXiv preprint arXiv:1910.03193}, 2019.

\bibitem[LKA{\etalchar{+}}20]{li2020fourier}
Zongyi Li, Nikola Kovachki, Kamyar Azizzadenesheli, Burigede Liu, Kaushik
  Bhattacharya, Andrew Stuart, and Anima Anandkumar.
\newblock Fourier neural operator for parametric partial differential
  equations.
\newblock {\em arXiv preprint arXiv:2010.08895}, 2020.

\bibitem[LL22]{lu2022priori}
Jianfeng Lu and Yulong Lu.
\newblock A priori generalization error analysis of two-layer neural networks
  for solving high dimensional schr{\"o}dinger eigenvalue problems.
\newblock {\em Communications of the American Mathematical Society},
  2(1):1--21, 2022.

\bibitem[LLW21]{lu2021priori}
Yulong Lu, Jianfeng Lu, and Min Wang.
\newblock A priori generalization analysis of the deep ritz method for solving
  high dimensional elliptic partial differential equations.
\newblock In {\em Conference on learning theory}, pages 3196--3241. PMLR, 2021.

\bibitem[LMW20]{li2020complexity}
Zhong Li, Chao Ma, and Lei Wu.
\newblock Complexity measures for neural networks with general activation
  functions using path-based norms.
\newblock {\em arXiv preprint arXiv:2009.06132}, 2020.

\bibitem[MEWW20]{ma2020towards}
Chao Ma, Weinan E, Stephan Wojtowytsch, and Lei Wu.
\newblock Towards a mathematical understanding of neural network-based machine
  learning: what we know and what we don't.
\newblock {\em arXiv preprint arXiv:2009.10713}, 2020.

\bibitem[MZ19]{muller2019deep}
Johannes M{\"u}ller and Marius Zeinhofer.
\newblock Deep ritz revisited.
\newblock {\em arXiv preprint arXiv:1912.03937}, 2019.

\bibitem[PW24]{park2024qualitative}
Josiah Park and Stephan Wojtowytsch.
\newblock Qualitative neural network approximation over $\mathbb r$ and
  $\mathbb c$: Elementary proofs for analytic and polynomial activation.
\newblock In {\em Explorations in the Mathematics of Data Science: The
  Inaugural Volume of the Center for Approximation and Mathematical Data
  Analytics}, pages 31--64. Springer, 2024.

\bibitem[RPK19]{raissi2019physics}
Maziar Raissi, Paris Perdikaris, and George~E Karniadakis.
\newblock Physics-informed neural networks: A deep learning framework for
  solving forward and inverse problems involving nonlinear partial differential
  equations.
\newblock {\em Journal of Computational physics}, 378:686--707, 2019.

\bibitem[SX22]{siegel2022high}
Jonathan~W Siegel and Jinchao Xu.
\newblock High-order approximation rates for shallow neural networks with
  cosine and {ReLU}$^k$ activation functions.
\newblock {\em Applied and Computational Harmonic Analysis}, 58:1--26, 2022.

\bibitem[SX23]{siegel2023characterization}
Jonathan~W Siegel and Jinchao Xu.
\newblock Characterization of the variation spaces corresponding to shallow
  neural networks.
\newblock {\em Constructive Approximation}, 57(3):1109--1132, 2023.

\bibitem[Woj22]{wojtowytsch2022optimal}
Stephan Wojtowytsch.
\newblock Optimal bump functions for shallow relu networks: Weight decay, depth
  separation and the curse of dimensionality.
\newblock {\em arXiv preprint arXiv:2209.01173}, 2022.

\bibitem[Woj25]{pde_relu}
Stephan Wojtowytsch.
\newblock A note on elliptic regularity theory in barron spaces.
\newblock {\em in preparation}, 2025.

\bibitem[Wu23]{wu2023embedding}
Lei Wu.
\newblock Embedding inequalities for barron-type spaces.
\newblock {\em arXiv preprint arXiv:2305.19082}, 2023.

\bibitem[WW20]{wojtowytsch2020can}
Stephan Wojtowytsch and E~Weinan.
\newblock Can shallow neural networks beat the curse of dimensionality? a mean
  field training perspective.
\newblock {\em IEEE Transactions on Artificial Intelligence}, 1(2):121--129,
  2020.

\bibitem[Y{\etalchar{+}}18]{yu2018deep}
Bing Yu et~al.
\newblock The deep ritz method: a deep learning-based numerical algorithm for
  solving variational problems.
\newblock {\em Communications in Mathematics and Statistics}, 6(1):1--12, 2018.

\end{thebibliography}

\end{document}